\documentclass[12pt]{amsart}

\usepackage{amsmath,amsfonts,amssymb,amscd,amsthm,latexsym,mathrsfs}

\usepackage{pstricks,pst-node,pst-plot}
\usepackage{graphicx}
\usepackage{pstricks,pst-node,pst-plot}

\usepackage{amsmath,amscd}
\usepackage{amsmath}
\usepackage{amsfonts}
\usepackage{amsthm}

\theoremstyle{plain}
\newtheorem{thm}{Theorem}[section]
\newtheorem{cor}[thm]{Corollary}
\newtheorem{lem}[thm]{Lemma}

\theoremstyle{definition}
\newtheorem{defi}[thm]{Definition}
\newtheorem{rem}[thm]{Remark}

\begin{document}

\title[Axiomatisability problems for $S$-Posets
 ]{Axiomatisability problems for $S$-Posets}
\subjclass{20 M 30, 03 C 60}
\date{\today}

\author{Victoria Gould}
\email{varg1@york.ac.uk}

\author{Lubna Shaheen}
\email{lls502@york.ac.uk}
\address{Department of Mathematics\\University
  of York\\Heslington\\York YO10 5DD\\UK}

\begin{abstract} Let $\mathcal{C}$ be a class of algebras of a given fixed type $\tau$. Associated with the type is a first order language $L_{\tau}$. One can then
ask the question, when
is the class $\mathcal{C}$ axiomatisable by sentences of $L_{\tau}$?
In this paper we will be considering axiomatisability 
problems for classes of left $S$-posets over a pomonoid $S$ (that is, a monoid $S$ equipped
with a partial order compatible with the  binary operation). We aim to determine the pomonoids $S$ such that certain categorically defined classes are axiomatisable.
The classes we
consider are the free $S$-posets, the projective $S$-posets and classes arising from flatness properties. Some of these cases have been studied in a recent article by Pervukhin and Stepanova. We present some general strategies to determine axiomatisability, from which their results for the classes of weakly po-flat and po-flat $S$-posets will follow. We also consider a number of classes not previously examined.
\end{abstract}

\maketitle

\section{Introduction and Preliminaries }\label{sec:intro}

A {\em pomonoid} is a monoid $S$ with a partial order $\leq$  which is compatible with the binary operation. Just as the representation of a monoid $M$ by mappings of sets gives us the theory of $M$-acts, representations of a pomonoid $S$ by order-preserving maps of partially ordered sets gives us $S$-posets. Thus a left $S$-poset is a non-empty partially ordered set
$A$ on which $S$-acts on the left, that is, 
there is a map $S\times A\rightarrow A$, where
$(s,a) \mapsto sa$ such that for all $s,t\in S$ and
$a\in A$,
\[s(t(a)) = (st)a\mbox{ and } 1  \, a = a\]
such that the map is monotone in both co-ordinates, that is, for all
$s,t\in S$ and $a,b\in S$ with $a\leq b$,
\[sa\leq ta\mbox{ and  }sa\leq sb.\]
The class of all left $S$-posets is denoted by $S$-Pos. It is worth pointing out
in this Introduction that $S$-posets (indeed, pomonoids) are not merely
algebras, they are relational structures. As such, care is needed to
take account of the partial order relation, particularly 
when considering congruences.

A morphism $\phi:A\to B$ from a left $S$-poset $A$ to a left $S$-poset $B$ is called an {\em $S$-poset morphism} 
or more briefly, {\em $S$-pomorphism},
 if it preserves the action of $S$ (that is, it is an $S$-act morphism) and the ordering on $A$. In other words, for all $a,b\in A$ with $a\leq b$ and $s\in S$ we have
 \[(as)\phi=a\phi s\mbox{ and }a\phi\leq b\phi.\]
It
  is  an {\em isomorphism} if, in addition, it is a bijection such that
  the inverse is also an $S$-pomorphism, that is, for all $a,b\in B$ with
  $a\leq b$ we have that  $ a\phi \leq  b\phi $ in A. We then say that $A$ and $B$ are {\em isomorphic} and write $A \cong B$. Note that a bijective $S$-pomorphism need not to be isomorphism.

We denote the category of left $S$-posets and $S$-pomorphisms by ${\mathbf S}$-{\bf Pos}. Dual definitions give us the class Pos-$S$ of
right $S$-posets and the corresponding notion of $S$-pomorphisms give
us  
the category {\bf Pos}-${\mathbf S}$ of right $S$-posets and $S$-pomorphisms. 

The study of $M$-acts over a monoid $M$ has been well established since the 1960s, and received a boost following the publication of the monograph \cite{kilpknauer} in 2000.
On the other hand, the investigation
 of $S$-posets, initiated by 
Fakhruddin in the 1980s \cite{fak:1986}, \cite{fak:1988}, was not taken up
again 
until this millenium, which has seen a burst of activity on this topic,
mostly (but not exclusively) 
concentrating on projectivity and various notions of flatness for $S$-poset, as we do here. Definitions and concepts relating to 
flatness are given in Section~\ref{sec:prelims}; an
excellent survey is given in \cite{syd:conf}.

Associated with the class $S$-Pos for a pomonoid $S$ we have a first order language
$L^{\leq}_S$, which has no constant symbols, a unary function symbol $\lambda_s$ for each
$s\in S$, and (other than $=$), a single relational symbol $\leq$ with
$\leq$ being binary. An $S$-poset provides an interpretation of $L^{\leq}_S$ in the
obvious way, indeed in $L^{\leq}_S$ we write $sx$ for $\lambda_s(x)$.
A class of  $\mathcal{C}$ of left $S$-posets is {\em axiomatisable} (or {\em elementary}) if there
is a set of sentences $\Pi$ of $L^{\leq}_S$ such that for any 
member $A$ of $\mathcal{C}$,
$A$ lies in $\mathcal{C}$ if and only if all sentences of $\Pi$ are true in
$A$, that is, $\mathcal{C}$ is a {\em model} of $\Pi$. We say in this case
that $\Pi$ {\em axiomatises} $\mathcal{C}$. We note that $S$-Pos
itself is axiomatisable amongst all interpretations of $L^{\leq}_S$. For any $s,t\in S$ 
and $u,v\in S$ with $u\leq v$ we define sentences
\[\varphi_{s,t}:= (\forall x)\big(s(t(x))=(st)x\big),\,
\theta_s:= (\forall x,y)\big(x\leq y\rightarrow sx\leq sy\big)
\mbox{ and }
\psi_{u,v}:= (\forall x)(ux\leq vx).\]
Then $\Pi_S$ axiomatises $S$-Pos where
\[\Pi_S=\{ (\forall x)(1\, x=x)\}\cup\big\{ \varphi_{s,t}:s,t\in S\big\}
\cup\big\{ \theta_s:s\in S\big\}\cup\big\{ \psi_{u,v}:u,v\in S,u\leq v\big\}.\]

Some  classes of left $S$-posets
are axiomatisable for {\em any} monoid $S$. For example, the class $\mathcal{T}$ of 
left $S$-posets with the trivial partial order is axiomatised by 
\[\Pi_S\cup\{ (\forall x,y)\big( x\leq y\rightarrow x=y\big)\}.\]
To save repetition, we will assume from now on that when axiomatising a class of
left $S$-posets, $\Pi_S$ is understood, so that we would say $\{ (\forall x,y)\big( x\leq y\rightarrow x=y\big)\}$ axiomatises $\mathcal{T}$.
 Other
natural classes of left $S$-posets are axiomatisable for some
pomonoids and not for others and it is our aim here to investigate the monoids that
arise.

Corresponding questions for classes of $M$-acts over a monoid $M$ have
been answered in \cite{gould, step, bulmanfleminggould}
and \cite{gould:tartu}, see also the survey article \cite{gouldpalyutin}. 
The classes of projective (strongly flat, po-flat, weakly po-flat) left $S$-posets $
\mathcal{P}r$($\mathcal{SF},\mathcal{PF},\mathcal{WPF}$) have recently been
considered in \cite{step:2009} (which uses slightly different
terminology;
the results also appearing in
\cite{shaheen:2010}) as has the class $\mathcal{F}r$ of
free left $S$-posets in the case where $S$ has only 
finitely many right ideals. We note that many of the 
techniques of \cite{step:2009} follow those in the $M$-act case 
and, for this reason, we aim here to produce two general strategies 
that will deal with a number of axiomatisability questions for classes 
of $S$-posets (and, with minor adjustment, $M$-acts). In particular they
may be applied to $\mathcal{PF}$ and $\mathcal{WPF}$. Just as many
concepts of flatness that are equivalent for $R$-modules over a unital ring $R$ are different for $M$-acts, so many concepts that coincide for $M$-acts split
for $S$-posets. Thus \cite{step:2009} left a number of classes open; 
we address many of them here, with both our general techniques and 
ad hoc methods.
 
The structure of the paper is as follows. After
Section~\ref{sec:prelims} which gives brief details of the concepts
required to follow this article, we present  in
Section~\ref{sec:tossings}
our {\em general} axiomatisability
results, which apply to various classes defined by  flatness
properties.
There are two kinds of results,
both phrased in terms of `replacement tossings'; we show how they 
may be applied to reproduce the results of \cite{step:2009}
determining for which pomonoids $\mathcal{PF}$ or
$\mathcal{PWF}$ are axiomatisable, together with a number of other applications. In
Section
~\ref{sec:specific} we then consider classes defined by  flatness
conditions that translate into  so called `interpolation conditions'. In these
cases we can give rather more direct arguments, avoiding the concept of
replacement tossing. Section~\ref{sec:projfree} briefly visits the
question of axiomatisability of $\mathcal{F}r$ and $\mathcal{P}r$; the
results here are easily deducible from the corresponding ones for
$M$-acts. Finally in Section~\ref{sec:open} we present some open problems.

\section{Preliminaries: flatness properties for $S$-posets}\label{sec:prelims}

Free and projective $S$-posets have the standard categorical
definitions.
We remark that \cite{step:2009} distinguishes between $S$-posets over a
pomonoid
$S$ that
are
free over posets and those free over sets: the free $S$-posets
we consider here are what \cite{step:2009} would refer
to as {\em free over sets}. The classes of free (projective) left
$S$-posets are denoted by $\mathcal{F}r$ ($\mathcal{P}r$), respectively.
The structure of $S$-posets in $\mathcal{F}r$  and $\mathcal{P}r$ is
transparent.

First note that for a symbol $x$ we let
$Sx=\{ sx \mid s\in S\}$ be a set of elements of $S$ such that  $Sx$ becomes a left $S$-poset (isomorphic to
$_SS$) if we define $s(tx)=(st)x$ for all $s,t\in S$ and $sx\leq tx$
if and only if $s\leq t$ in $S$.

\begin{thm}\label{thm:freeproj}\cite{shiliuwangbulmanfleming} $(i)$ An 
$S$-poset $A$ is free on a set
$X$ if and only if
$A\cong \bigcup_{x\in X}Sx$ where for all $x,y\in X$ and $s,t\in S$,
\[sx\leq ty\mbox{ if and only if }x=y\mbox{ and }s\leq t.\]
$(ii)$ An $S$-poset is projective if and only if it is isomorphic to a
disjoint union of incomparable $S$-posets of the form $Se$, where $e$ is idempotent.
\end{thm}

As in the unordered case, it is clear that
every free $S$-poset is projective and (provided $S$ has
idempotents other than $1$, the converse
is not true.

To define notions of flatness, we need to consider the tensor products  of
$S$-posets.  Let $A$ be a right $S$-poset and $B$ a left $S$-poset.
The tensor
 product, which is denoted by $A \otimes B$,
 is the quotient of  $A \times B$, which considered as an $S$-poset under trivial $S$-action, 
 by the order congruence relation $\theta$ on $A \times B$ generated by 
\[\{(as,b),(a,sb): s \in S, a \in A ,b \in B \}.\]
We will denote the equivalence class of $ (a,b)  \in A \times B$ 
with respect to congruence $\theta$ by $ a \otimes b$.
We say a little more about order congruences in Section~\ref{sec:tossings}.
   The following lemma explains the ordering in $A \otimes B$.

\begin{lem}\label{lem:ptossing}\cite{shiliuwangbulmanfleming} Let 
$S$ be a pomonoid, let 
$A$ be a right $S$-poset, $B$ a left $S$-poset, $a , a' \in A$, and $b,b' 
\in B$. Then $ a \,\otimes \,b \leq a' \,\otimes \, b' $ in $A \otimes B$ if and only if there exists $a_2,a_3,\cdots,a_m \in A$, $b_1,b_2,\cdots,b_m \in B$ and $s_1,t_1,\cdots,s_m,t_m \in S$ such that 
\[\begin{array}{rclrcl}
                &    &                    & b& \leq & s_1 b_1\\
                 a s_1 & \leq & a_2 t_1        & t_1 b_1 & \leq & s_2 b_2\\   
                a_2 s_2 & \leq & a_3 t_2     &  t_2 b_2 & \leq & s_3 b_3 \\
              &  \vdots & &&\vdots&  \\
 a_m s_m & \leq & a' t_m          & t_m b_m & \leq & b' 
\end {array} \]

\noindent It follows that $ a' \, \otimes b^{'} \leq a \, \otimes \, b $ if and only if there exists $ c_2 , \cdots c_n \in A $ and $d_1 ,\cdots ,d_n \in B$ and $ u_1, v_1, \cdots , u_n, v_n  \in S$ such that  

\[\begin{array}{rclrcl}

               &    &                    & b'& \leq & u_1 d_1\\
              a' u_1  & \leq & c_2 v_1          & v_1 d_1 & \leq & u_2 d_2 \\
              c_2 u_2 & \leq & c_3 v_2         & v_2 d_2 & \leq & u_3 d_3    \\
             &  \vdots &  && \vdots& \\
              c_n u_n & \leq & a v_n          & v_n d_n & \leq & b
\end{array}\]
Thus $ a \otimes b = a' \otimes b' $ in $A \otimes B$ if and only if 
$(*)$ and $(**)$ exist.

 \end{lem}

\begin{defi}\label{ptossing}
The sequence $(*)$ is called an {\em ordered tossing} $\mathcal{T}$
of length $m$
from $(a,b)$ to $(a',b')$. The {\em ordered skeleton}
of $\mathcal{S(T)}$ is the sequence $\mathcal{S(T)}
=(s_1,t_1,\cdots,s_m,t_m)$.
The two sequences  $(*)$ and $(**)$  constitute a 
{\em double ordered tossing} $\mathcal{DT}$ of length $m+n$, 
from $(a,b)$ to $(a',b')$
with {\em double ordered skeleton} $$\mathcal{S(DT)}= 
(s_1,t_1,\cdots,s_m,t_m,u_1,v_1,\cdots,u_n,v_n).$$   
We may also write $\mathcal{S(DT)}=(\mathcal{S}_1,\mathcal{S}_2)$
where \[\mathcal{S}_1=(s_1,t_1,\cdots,s_m,t_m)\mbox{  and }
\mathcal{S}_2=(u_1,v_1,\hdots, u_n,v_n).\]

\end{defi}

As in the case of $M$-acts different notions of flatness are 
drawn from the tensor functor 
\[-\otimes B : \mathbf{Pos}\mbox{{-\bf
S}} \to\mathbf{ Pos}\]
where if $A, A'$ are right $S$-posets and $f:A\rightarrow A'$ is
a pomorphism, 
\[A\mapsto A\otimes B
\mbox{ and }f\mapsto f \otimes I_{B}\]
and $$(a \otimes b)(f \otimes I_{B})=af\otimes b. $$

\begin{defi} An $S$-pomorphism $f :A \to B$ between two left $S$-posets $A$ and $B$ is called an {\em embedding} if it satisfies the condition
$$a \leq a' \Leftrightarrow  af \leq a'f.$$
\end{defi}

Elementary considerations of partially ordered sets (regarded
as $S$-acts over a trivial pomonoid) tell us that
 monomorphisms and embeddings in $\mathbf{S}$-{\bf Pos},
and indeed bijections and isomorphisms, are not the same. This leads us
to two variations on notions of flatness.

An $S$-poset $A$ 
is called {\em flat} if the functor $ - \otimes B$ takes embeddings
 in the category of $\mathbf{Pos}$-{\bf S} to one-one maps
 in the category {\bf Pos} of posets. 
It is called { \em (principally) weakly flat } if the functor $ - \otimes B$ 
takes embeddings of (principal) right ideals of $S$ into $S$
 to one-one maps in the category {\bf Pos}. A left $S$-poset $B$ is 
called {\em strongly flat} if the functor $- \otimes B$ preserves 
subpullbacks and subequalizers 
or equivalently \cite{bulmanlaan} if $B$ satisfies Condition 
$(P)$ and Condition $(E)$ which are defined as follows:

\noindent Condition (P): for all $ b,b' \in B$ and $ s,s' \in S$ if $ s \, b \leq s' \,b' $ then there exists $ b'' \in B$ and $u,u' \in S$ such that $b= u \, b'' ,\, b' = u' \, b''$ and $ s\,u \leq s'\,u'$; 

\noindent Condition (E): for all $ b \in B$ and $ s,s' \in S$ if 
$ s \, b \leq s' \, b $ then there exists $ b'' 
\in B$ and $ u \in S$ such that $ b = u \,b''$ and $ s \,u  \leq s' \,
u.$ \\
Such flatness conditions, i.e. using elements of $S$ and $S$-posets
rather than tossings explictly, we call {\em interpolation conditions}.
Weaker than either (P) or (E) we have\\
Condition (EP): for all $b\in B$ and $s,s'\in S$, if $sb\leq s'b$ then
there
exists $b''\in B$ and $u,u'\in S$ such that $b=ub''=u'b''$ and 
$su\leq s'u'$. The unordered version of this condition was introduced
for $M$-acts in \cite{golchin:2007}.

In \cite {shi}  Shi defined notions of po-flat, weakly po-flat,
 prinicipally weakly po-flat $S$-posets, as follows:\\
 an $S$-poset $B$ is called {\em po-flat} if the
 functor $- \otimes B$ takes embeddings
 in the category of $\mathbf{Pos}$-{\bf S} 
to embeddings in  $\mathbf{Pos}$. It is  {\em (principally) weakly po-flat}
 if the functor
 $ - \otimes B$ preserves the embeddings of (principal) right ideals
 of $S$ into $S$.

In the theory of $M$-acts over a monoid $M$, it is
true that  all $M$-acts satisfy the unordered version
of Condition (P) if and only 
if $M$ is a group. We can, however,   find an $S$-poset
over an ordered group $S$ which 
does not satisfy Condition (P). With this in mind,
Shi \cite{shi} defined  another notion for a left $S$-poset $B$
similar to Condition (P), called Condition (P$_w$):\\ 
Condition (P$_w$): for all  $b,b' \in B$ and 
$s,s'\in S$ if $ s \,b \leq s'\, b'$ then there exists
 $b'' \in B$, $ u,u' \in S$ such that $ s \,u \leq s'\, u' ,\, 
b \leq u \, b'',\,\, u' \,b'' \leq b'. $

Further, let $G$ be an ordered group, then all $G$-posets satisfy 
Condition (P$_w$) \cite{shi}. Clearly (P) implies (P$_w$) and from
\cite{shi},
(P$_w$) implies po-flat.

Shi \cite{shi} has shown that a left $S$-poset $B$ is weakly po-flat 
if and only if it is principally weakly po-flat and 
satisfies:\\
 Condition (W): for any $ b,b' \in B$ and $ s,s' \in S$,
if $s b \leq s' b' $ then implies that there exists $ b'' \in B ,\,p
 \in sS, p' \in s'S$ such that $ p \leq p',\, sb \leq p b'' ,\,p' b'' 
\leq s' b'$. 
Shi's proof is along the same
lines as that for 
$M$-acts by Syd Bulman-Fleming and McDowell in \cite{kilpknauer}. , who
have proved that a left $S$-act $A$ is weakly flat if and only if it is principally weakly flat and satisfies a condition analogous to Condition $(W)$
for $S$-acts. A proof analogous to those in \cite{kilpknauer,shi} gives
the following.

\begin{lem}\label{lem:condu} Let $S$ be a pomonoid. A left $S$-poset $B$
is
weakly flat if and only if it is principally weakly flat and
satisfies:\\
Condition (U): for all $b,b'\in B$ and $s,s'\in S$, if
$sb=sb'$ then there exists $b''\in B$, $p\in sS,p'\in s'S$, with $p\leq
p'$ and $sb=pb''=p'b''=s'b'$.
\end{lem}

We will denote the classes strongly flat, flat, weakly flat, 
principally weakly flat, po-flat, weakly po-flat, principally 
weakly po-flat left $S$-posets by
$$\mathcal{SF},\,\mathcal{F},\,\mathcal{WF},
\,\mathcal{PWF},\,\mathcal{PF},\,\mathcal{WPF},\,\mathcal{PWPF}$$ 
respectively. We will denote the classes of left $S$-posets satisfying
Conditions (P),   (E),   (EP),  (P$_w$),(W) and (U) by
\[\mathcal{P,E,EP,P}_w,\mathcal{W}\mbox{ and }\mathcal{U}.\]

Finally in our list of flatness properties
we turn out attention to those introducted in  \cite{golchin} 
by Golchin and Rezaei. They define Conditions (WP),(WP$_w$),(PWP)
and (PWP$_w$) for $S$-posets, which are derived from the concepts 
of subpullback diagrams in $\mathbf{S}$-{\bf Pos}.
 For details relating to subpullback diagrams in the category of 
$\mathbf{S}$-{\bf Pos} we refer the reader to \cite{golchin}. For our
purposes here it is enough to define (PWP) and (PWP$_w$) for a
left $S$-poset $B$:\\
Condition (PWP):  for all $b,b'\in B$ and $s\in S$,
if $sb\leq sb'$ then there exits $u,u' \in S$ and $b''\in B$
 such that $ b= u b'' \wedge b' = u' b''$ and $su \leq su'$;\\
Condition (PWP$_{w}$): for all $b,b '\in B$
and $s\in S$, if  $sb \leq sb'$ then there exist $u,u' \in S$ and $b''\in B$ such that 
$ b \leq  u b'' \wedge  u' b'' \leq b'$ and $su \leq su'$.

We denote by
\[\mathcal{WP,WP}_w,\mathcal{PWP}\mbox{ and }\mathcal{PWP}_w\]
the classes of left $S$-posets satisfying Conditions   
 (WP),(WP$_w$),(PWP)
and (PWP$_w$), respectively.

\begin{rem}\label{rem:table} \cite{golchin}\cite{shiliuwangbulmanfleming}
In $\mathbf{S}$-{\bf POS} we have the following implications, all of
which are known to be strict except for Condition (P$_w$) implies
po-flat:

\[\begin{array}{ccccccccccc}
 \mathcal{F}r& \Rightarrow &\mathcal{P}r&\Rightarrow &
\mathcal{SF}&\Rightarrow   &  \mathcal{P}          & \Rightarrow&
\mathcal{WP}
        &\Rightarrow &\mathcal{PWP}\\
& &&& &  &\Downarrow & & \Downarrow          &  & \Downarrow\\
 &   &&&  && \mathcal{P}_w   &\Rightarrow &\mathcal{WP}_w    &
\Rightarrow&
\mathcal{PWP}_{w}     \\ 
 &  &&&    &  & \Downarrow  &&\Downarrow    &          & \Downarrow \\
 &    &&&   &   & \mathcal{PF}&\Rightarrow &\mathcal{WPF} &\Rightarrow&
 \mathcal{PWPF} \\
 &  &&&    &    & \Downarrow & &\Downarrow             & & \Downarrow\\ 
   &     &&&   & &\mathcal{F} & \Rightarrow&\mathcal{WF} &\Rightarrow& 
\mathcal{PWF}  \end{array}\]

\end{rem}

We are interested in determining for which pomonoids are these classes
axiomatisable. Our major tool is that of an {\em ultraproduct}; further
details may be found in \cite{ck}.
The next result is crucial. 

\begin{thm}{\em(\L os's Theorem)}\label{thm:los}\cite{ck}
Let $L$ be a first order language, and let $\mathcal{C}$ be
a class of  $L$-structures. If $\mathcal{C}$ is
axiomatisable, then $\mathcal{C}$ is closed under ultraproducts.

\end{thm}

We now introduce a new notion of flatness, that can be adapted to many
of the classes above. 
Let ${\mathcal{C}}$ be a class of embeddings of right $S$-posets.
For example, $\mathcal{C}$ could be all embeddings, or all  embeddings of
right ideals into $S$ via inclusion maps. We say that a left
$S$-poset $B$ is {\em $\mathcal{C}$-flat} if the functor
 $-\otimes B$ maps every embedding $\mu:A\rightarrow A'\in
\mathcal{C}$ to a one-one  map 
$\mu\otimes I_B:A\otimes B\rightarrow A'\otimes B$. 
The class of $\mathcal{C}$-flat left $S$-posets 
is denoted by $\mathcal{CF}$. Similarly, if
$-\otimes B$ maps every embedding $\mu:A\rightarrow A'\in
\mathcal{C}$ to an embedding 
$\mu\otimes I_B:A\otimes B\rightarrow A'\otimes B$, then
we say that $B$ is {\em $\mathcal{C}$-po-flat} and we denote
the class of $\mathcal{C}$-po-flat left $S$-posets by $\mathcal{CPF}$.
Thus, if $\mathcal{C}$ is the
class of all embeddings of right $S$-posets, then
$\mathcal{CF}=\mathcal{F}$ and $\mathcal{CPF}=\mathcal{PF}$.

\section{ Axiomatisability of $\mathcal{CF}$ }\label{sec:tossings}

We describe our two general results involving `replacement tossings'. The first
characterise those pomonoids $S$ such that $\mathcal{CF}$
is axiomatisable, for a class $\mathcal{C}$ of right $S$-poset embeddings,
where $\mathcal{C}$ satisfies Condition (Free). This will enable us
to specialise to the case where $\mathcal{C}$ is the class of all 
right $S$-poset embeddings. For the second we consider an arbitary
class $\mathcal{C}$; we then specialise to the cases
where $\mathcal{C}$ consists of all inclusions
of (principal) right ideals into $S$.  We remark that similar methods can be applied
to axiomatisability problems for $S$-acts over a monoid $S$, as shown in 
\cite{shaheen:2010}.

\subsection{Axiomatisability of $\mathcal{CF}$  with Condition (Free)}\label{subsec:generala}\phantom{hhhhhhhhhhhhhhhhhhhhhhhhhhhhhhhhhhhhhhhhhhhhhhh}

It is convenient to introduce some notation. Let 

$$\mathcal{S}=(s_1,t_1,\hdots ,s_m,t_m)$$ 
be an ordered skeleton of length $m$. 

We define a formula 
 $\epsilon_{\mathcal{S}}$ of $R^{\leq}_{\mathcal{S}}$,
 where $R^{\leq}_{\mathcal{S}}$ is the first order language
 associated with right $S$-posets, as follows:
\[
\epsilon_{\mathcal{S}}(x,x_2,\cdots,x_m,x'):=
\big( xs_1  \leq x_2 t_1 \,\wedge\, x_2 s_2  \leq x_3 t_2 \,\wedge\, \hdots \,\wedge\, x_m s_m \leq x' t_m  \big) \]
and a formula $\theta_{\mathcal{S}}$ of $L^{\leq}_S$ by
\[\theta_{\mathcal{S}}(x,x_1,\cdots,x_m,x'):= \big(   x \leq s_1 x_1 \, \wedge\, t_1 x_1 \leq s_2 x_2\, \wedge \, \hdots \, \wedge\, t_m x_m \leq x'  \big) .\]
Suppose now that 

$$\mathcal{S}=(\mathcal{S}_1,\mathcal{S}_2)=(s_1,t_1,\hdots ,s_m,t_m, u_1,v_1,\hdots ,u_n,v_n)$$ 
is  a  double ordered skeleton of length $m+n$. 
We put
\[\delta_{\mathcal{S}}(x,x'):= 
(\exists x_2 \hdots \exists x_m \exists y_2\hdots \exists y_n ) 
\epsilon_{\mathcal{S}_1}(x,x_2,\hdots ,x_m,x')\,\wedge\,
\epsilon_{\mathcal{S}_2}(x',y_2,\hdots ,y_n,x).\]

On the other hand we define the formula 
$$\gamma_{\mathcal{S}}(x,x'):=(\exists x_1 \cdots \exists x_m\exists y_1 \cdots \exists y_n) \theta_{\mathcal{S}_1}(x,x_1,\hdots , x_m,x')
\, \wedge \, \theta_{\mathcal{S}_2}(x',y_1,\hdots ,y_n,x').$$

\begin{rem}\label{rem:formulas} Let $A,B$ be right and left $S$-posets, respectively,
let $a,a'\in A$ and $b,b'\in B$. 

$(i)$ The pair $(a,b)$ is connected to the pair $(a',b')$
via a double ordered tossing with double ordered skeleton $\mathcal{S}$ if and only if
$\delta_{\mathcal{S}}(a,a')$ is true in $A$ and
$\gamma_{\mathcal{S}}(b,b')$ is true in $B$.

$(ii)$ If $\delta_{\mathcal{S}}(a,a')$  is true in $A$ and $\psi:A\rightarrow A'$ is
a (right) $S$-pomorphism, then $\delta_{\mathcal{S}}(a\psi,a'\psi)$ is true
in $A'$.

$(iii)$ If   $\gamma_{\mathcal{S}}(b,b')$ is true in $B$ and $\tau:B\rightarrow B'$ is
an $S$-pomorphism, then $\gamma_{\mathcal{S}}(b\tau,b'\tau)$ is true in
$B\tau$.
\end{rem}

\begin{defi} We say that $\mathcal{C}$ {\em satisfies Condition (Free)} if for each double ordered skeleton $\mathcal{S}$ there is an embedding $\tau_{\mathcal{S}}:W_{\mathcal{S}} \to W^{'}_{\mathcal{S}}$ in $\mathcal{C}$ and $u_{\mathcal{S}},\,u^{'}_{\mathcal{S}} \in W_{\mathcal{S}}$ such that $\delta_{\mathcal{S}}(u_{\mathcal{S}}\tau_{\mathcal{S}},u^{'}_{\mathcal{S}}\tau_{\mathcal{S}})$ is true in $W^{'}_{\mathcal{S}}$ and further for any embedding $\mu:A \to A^{'} \in \mathcal{C}$  and any $a,a' \in A$  such that $\delta_{\mathcal{S}}(a \mu, a' \mu )$ is true in $A'$ there is a morphism $\nu:W^{'}_{\mathcal{S}} \to A'$ such that $u_{\mathcal{S}}\tau_{\mathcal{S}}\nu = a \mu,\,u^{'}_{\mathcal{S}}\tau_{\mathcal{S}}\nu = a' \mu $ and $W_{\mathcal{S}} \tau_{\mathcal{S}} \nu \subseteq A \mu.$

\end{defi}

\begin{lem}\label{lem:coflat} Let $\mathcal{C}$ be a class of embeddings of right $S$-posets satisfying Condition (Free). Then the following are equivalent for a left $S$-poset $B$:

(i) $B$ is $\mathcal{C}$-flat;

 (ii) $-\otimes B$ maps the embeddings  $\nu_{\mathcal{S}}:W_{\mathcal{S}} \to W^{'}_{\mathcal{S}}$ in the category {\bf Pos-}${\mathbf S}$ to monomorphisms in the category of ${\bf Pos}$, for every double ordered skeleton $\mathcal S$;

(iii) if $(\mu_{\mathcal{S}} \tau_{\mathcal{S}},b)$ and $(\mu^{'}_{\mathcal{S}} \tau_{\mathcal{S}},b')$ are connected by a double ordered tossing over $W^{'}_{\mathcal{S}}$ and $B$ with double ordered skeleton $\mathcal{S}$, then $(u_{\mathcal{S}},b)$ and $(u^{'}_{\mathcal{S}},b')$ are connected by a double ordered tossing over $W_{\mathcal{S}}$ and $B$.

\end{lem}
\begin{proof}
Clearly we need only show that $(iii)$ implies $(i)$.
Suppose that $(iii)$ holds, let $\mu:A \to A'$ lie in $\mathcal{C}$ and suppose that $$(a \mu ,b),(a' \mu ,b') \in A' \times B$$ 
 are connected via a double ordered tossing with double ordered skeleton $\mathcal{S}$,
 so that $\gamma_{\mathcal{S}}(b,b')$ holds. From considering the left hand side of the double ordered tossing, we have that $\delta_{\mathcal{S}}(a \mu, a' \mu)$ is true in $A'$. By assumption there is an embedding $ \tau_{\mathcal{S}}:W_{\mathcal{S}} \to W'_{\mathcal{S}}$ in $\mathcal{C}$  and $ u_{\mathcal{S}}, u'_{\mathcal{S}} \in W_{\mathcal{S}}$ such that $ \delta_{\mathcal{S}}(u_{\mathcal{S}} \tau_{\mathcal{S}},u'_{\mathcal{S}} \tau_{\mathcal{S}})$ is true in $W'_{\mathcal{S}}$, and a morphism $\nu:W'_{\mathcal{S}} \to A'$ such that $u_{\mathcal{S}} \,\tau_{\mathcal{S}} \nu= a \mu , \, u'_{\mathcal{S}} \,\tau_{\mathcal{S}} \nu= a' \mu $ and $W_{\mathcal{S}} \tau_{\mathcal{S}} \nu \subseteq A \mu $. Since $\delta_{\mathcal{S}}(u_{\mathcal{S}} \tau_{\mathcal{S}},u'_{\mathcal{S}} \tau_{\mathcal{S}})$ is true in $W'_{\mathcal{S}}$, there is a double ordered tossing from $(u_{\mathcal{S}} \tau_{\mathcal{S}},b)$ to $(u'_{\mathcal{S}} \tau_{\mathcal{S}},b')$ over $W'_{\mathcal{S}}$ and $B$, with double ordered skeleton $\mathcal{S}$. From $(iii)$, it follows that $(u_{\mathcal{S}},b)$ and $(u'_{\mathcal{S}},b')$ are connected via a double ordered tossing over $W_{\mathcal{S}}$ and $B$ with double ordered skeleton $\mathcal{T}$ say. It follows that $\delta_{\mathcal{T}}(u_{\mathcal{S}},u'_{\mathcal{S}})$ is true in $W_{\mathcal{S}}$ and so $\delta_{\mathcal{T}}(u_{\mathcal{S}} \tau_{\mathcal{S}} \nu,u'_{\mathcal{S}}\tau_{\mathcal{S}} \nu )$, that is, $\delta_{\mathcal{T}}(a \mu, a' \mu)$ is true in $A \mu$. Since $\mu $ is an ordered  embedding we deduce that $\delta_{\mathcal{T}}(a,a')$ is true in $A$ and consequently, $(a,b)$ and $(a',b')$ are connected via a double ordered tossing with double ordered skeleton $\mathcal{T}$ over $A$ and $B$. Hence $B$ is $\mathcal{C}$-flat as required.
\end{proof} 

Our next aim is to show that the class of all
embeddings of right $S$-posets has Condition (Free). To this
end we present a `Finitely Presented Flatness Lemma' for $S$-posets.
First, a {\em po-congruence} on a left $S$-poset $B$  is an 
equivalence relation $\rho$ which is compatible with the action on $S$, 
such that in addition $B/\rho$ may be a partially ordered in a 
way that the natural map $B\to B/\rho$ is $S$-pomorphism. For further details
concerning congruences on ordered algebras, we refer the reader to \cite{cl} and for the specific case of $S$-posets, to \cite{yunshi}. Given
a subset $R$ of $B\times B$, it is possible to construct 
a po-congruence $\equiv_R$ on $B$ such that $[a]\preceq_R [b]$
for every
 $(a,b)\in R$, 
where $\preceq_R$ is the ordering in $B/\equiv_R$, and is such that
if $\alpha:B\rightarrow C$ is an $S$-pomorphism from $B$ to
any left $S$-poset $C$ with $a\alpha\leq b\alpha$ for all $(a,b)\in R$,
then there exists a pomorphism $\beta:B/\equiv_R\rightarrow C$ such
that $[b]\beta=b\alpha$, for all $b\in B$.

 For a double ordered skeleton $\mathcal{S}=(\mathcal{S}_1,\mathcal{S}_2)$
 where
 \[\mathcal{S}_1=(s_1,t_1,\cdots,s_m,t_m)
 \mbox{ and }\mathcal{S}_2=(u_1,v_1,\cdots,u_n,v_n),\]
  we let  $F^{m+n}$ be the free right $S$-poset
 $$ xS \,\dot{\cup} \,x_2 S \,\dot{\cup} \,\hdots  x_m S \, \dot{\cup} \, y_2 S \,\dot{\cup}\, y_3 S  \hdots \, y_n S \, \dot{\cup} \, x' S$$
and let $R_{\mathcal{S}}$ be the set
 \[ \big \{(xs_1,x_2t_1),(x_2 s_2, x_3 t_2), \hdots, (x_m t_{m-1}),(x_m s_m , x' t_m),\]
 \[(x'u_1,y_2 v_1),(y_2 u_2,y_3 v_2), \hdots, (y_n u_n, x v_n) \big \}.\]

We aim here to define a least ordered congurence relation which contains a relation $H \subseteq A \times A.$

Let us abbreviate by 
$ \equiv_{\mathcal S}$ the $S$-poset congruence  
 $\equiv_{R_{\mathcal S}} $ induced by $R_{\mathcal{S}}$.
We abbreviate  the order $\preceq_{R_{\mathcal{S}}}$ on $F^{m+n}/\equiv_{\mathcal{S}}$
 by $\preceq_{\mathcal{S}}$.

 If $B$ is a left $S$-poset and 
$b, \, b_1, \, , \cdots , b_m ,\, d_1, \, d_2 , \, , \cdots , d_n,
 \, b' \in B$ are such that 
 \[\theta_{\mathcal{S}_1}(b,b_1,\hdots, b_m,b')\mbox{
  and }
  \theta_{\mathcal{S}_2}(b',d_1,\hdots ,d_n,b)\]
  hold,
 then the double ordered tossing

$$
\begin{array}{rclcrcl}
                      &    &                    & b& \leq & s_1 b_1\\\
                [x]s_1 & \leq & [x_2] t_1        & t_1 b_1 & \leq & s_2 b_2\\\
                                
                [x_2] s_2 & \leq & [x_3] t_2     &  t_2 b_2 & \leq & s_3 b_3 \\
                
              &  \vdots &                         && \vdots & \\\
                  
               [ x_m ]s_m & \leq & [x']t_m          & t_m b_m & \leq & b' \\
              
              & \vspace{2mm} &                   & \vspace{2mm} & \\
               &    &                    & b'& \leq & u_1 d_1\\\
              [x']u_1  & \leq & [y_2]v_1          & v_1 d_1 & \leq & u_2 d_2 \\\
              [y_2]u_2 & \leq & [y_3]v_2         & v_2 d_2 & \leq & u_3 d_3    \\
              &  \vdots &                        & &\vdots & \\\
             [y_n] u_n & \leq & [x] v_n          & v_n d_n & \leq & b
\end{array}
$$
  over $F^{m+n} / \equiv_{\mathcal S}$ and $B$ is called a
 {\it double ordered standard tossing}; clearly it has double ordered skeleton $\mathcal {S}.$

It is clear that (by considering a trivial left $S$-poset $B$), 
the set of all double ordered skeletons $\mathbb{DOS}$ is the set of all finite even length
sequences of elements of $S$, of length at least 4.

\begin{lem}\label{lem:fpflat} The following conditions are  equivalent for a left $S$-poset $B$:

(i) $B$ is flat;

(ii) $-\otimes B$ maps  embeddings of $[x]S \cup [x']S$ into $F^{m+n}/ \equiv_{\mathcal S}$ in the category {\bf Pos-}${\mathbf S}$ to monomorphisms in the category of ${\bf POS}$, for every double ordered skeleton $\mathcal S$;

(iii)  if $\big([x],b \big)$ and $\big([x'],b' \big)$ are connected by a double ordered standard tossing over $F^{m+n}/ \equiv_{\mathcal S}$ and $B$ (with double ordered skeleton  $\mathcal S$), then they are connected by a double ordered tossing over $[x]S \cup [x']S$ and $B$.
\end{lem}
\begin{proof}
We will prove here only $(iii) \Rightarrow (i)$.
Suppose that $B$ satisfies condition $(iii)$, let $a , a'$ belongs to
any
 right $S$-poset $A$, let $b,b' \in B$, and suppose that
 $ a \otimes b=a' \otimes b'$ in $A \otimes B$ via a double ordered tossing
 with double ordered skeleton $\mathcal{S}=(\mathcal{S}_1,\mathcal{S}_2)$, 
 where $\mathcal{S}_1,\mathcal{S}_2$ have lengths $m$ and
 $n$, respectively. By Remark~\ref{rem:formulas}, $\delta_{\mathcal{S}}(a,a')$ is true in $A$ and
 $\gamma_{\mathcal{S}}(b,b')$ is true in $B$. Since $\delta_{\mathcal{S}}([x],[x'])$
 holds in $F^{m+n}/ \equiv_{\mathcal S}$, we have that 
 $\big([x],b \big)$ and $\big([x'],b' \big)$ are connected by a double ordered standard tossing over $F^{m+n}/ \equiv_{\mathcal S}$ and $B$. 
By the given hypothesis we have that 
$\big([x],b \big)$ and $\big([x'],b' \big)$  
connected via a double ordered tossing in $\big( [x]S \cup [x']S \big ) \otimes B$, say with
 double ordered skeleton $\mathcal{U}$.

Since $\delta_{\mathcal{S}}(a,a')$ is true in $A$, there
are elements $a_2,\hdots, a_m,c_2,\hdots ,c_n\in A$ such that
\[\epsilon_{\mathcal{S}_1}(a,a_2,\hdots ,a_m,a')\mbox{ and }
\epsilon_{\mathcal{S}_2}(a',c_2,\hdots ,c_n,a)\]
hold in $A$. Let $ \phi:F^{m+n} \to A $ be the $S$-pomorphism which is defined by 
$x\phi= a,\, x_i\phi = a_i $\, $( 2 \leq i \leq m)$,\, $x' \phi=a'$
 and $y_j\phi=c_j \,( 2 \leq j \leq n)$. 
Since $ u\phi \leq u'\phi  $ for all $ ( u, u') \in 
R_ {\mathcal S} $,  we have
that $ \overline{\phi} : F^{m+n} / 
\equiv_{\mathcal S}  \to A $ given by
$[z]\overline{\phi}=z\phi$ is a well defined $S$-pomorphism.
We have that $\delta_{\mathcal{U}}([x],[x'])$ holds in $[x]S\cup [x']S$, so that by
Remark~\ref{rem:formulas}, $\delta_{\mathcal{U}}(a,a')$ holds in $aS\cup a'S$. Since
also $\gamma_{\mathcal{U}}(b,b')$ holds in $B$, we have that
$(a,a')$ and $(b,b')$ are connected by a double ordered tossing over
$aS\cup a'S$ and $B$, so that $a\otimes b=a'\otimes b'$ in $aS\cup a'S\otimes B$.
Thus $B$ is flat, as required.
\end{proof}

With a similar argument, we prove the following.

\begin{lem}\label{lem:allfree} The class Pos-$S$ of all right $S$-posets has Condition (Free).
\end{lem}
\begin{proof} Let $\mathcal{S}$ be a double ordered skeleton of
length $m+n$,
let $W'_{\mathcal{S}}=F^{m+n}/\equiv_{\mathcal{S}}$,
$W_{\mathcal{S}}=[x]S\cup [x']S$ and let
$\tau_{\mathcal{S}}:W_{\mathcal{S}}\rightarrow W'_{\mathcal{S}}$ denote
inclusion. Then $[x],[x']\in W_{\mathcal{S}}$ and
$\delta_{\mathcal{S}}([x]\tau_{\mathcal{S}},[x']\tau_{\mathcal{S}})$
is true in $W'_{\mathcal{S}}$. 

Let $\mu:A\rightarrow A'$ be any right $S$-poset
embedding such that $\delta_{\mathcal{S}}(a\mu,a'\mu)$ holds in $A'$, for
some $a,a'\in A$. As in Lemma~\ref{lem:fpflat}, there is as a consequence
an $S$-pomorphism $\nu:W'_{\mathcal{S}}\rightarrow A'$ such that
$[x]\tau_{\mathcal{S}}\nu=a\mu$ and
$[x']\tau_{\mathcal{S}}\nu=a'\mu$. Clearly 
\[W_{\mathcal{S}}\tau_{\mathcal{S}}\nu=
([x]S\cup [x']S)\tau_{\mathcal{S}}\nu=[x]\tau_{\mathcal{S}}\nu S \cup [x']\tau_{\mathcal{S}}\nu S=
a\mu S\cup a'\mu S=
(aS\cup a'S)\mu\subseteq A\mu.\]
Thus, with $u_{\mathcal{S}}=[x]$ and $u'_{\mathcal{S}}=[x']$,
we see that Condition (Free) holds.

\end{proof}

Let $\mathcal{C}$ be a class of ordered embeddings of right $S$-posets. Let $\overline{\mathcal{C}}$ be the set of products of morphisms in $\mathcal{C}$
(with  the obvious definition and pointwise ordering). 

\begin{lem}\label{lem:prodC} Let $\mathcal{C}$ be a class of embeddings of right $S$-posets, satisfying Condition (Free).  If a left $S$-poset $B$ is $\mathcal{C}$-flat, then it is $\overline{\mathcal{C}}$-flat.
\end{lem}
\begin{proof}
Let $I$ be an indexing set and let $\gamma_i :A_i \to A'_i \in \mathcal{C}$ for all $i \in I$. Let $A=\prod_{i \in I} A_i, \, \,A'=\prod_{i \in I}A'_{i}$ and let $\gamma:A\rightarrow A'$ be the canonical embedding, so that
$(a_i)\gamma=(a_i\gamma_i)$.

 Suppose $B$ is a $\mathcal{C}$-flat left $S$-poset. Let 
${\underline{a}}=(a_i),
{\underline{a}'}=(a_i') \in A$ and $b,b'\in B$ be
such that $\underline{a} \gamma \otimes b = \underline{a}'\gamma \otimes b$ in $A' \otimes B$. Then for some double
ordered skeleton $\mathcal{S}$,
\[
A'\models \delta_{\mathcal{S}}(\underline{a}\gamma,\underline{a'}\gamma)
\mbox{ and }B\models \gamma_{\mathcal{S}}(b,b').\]
It follows that for each $i\in I$,
\[A_i\models \delta_{\mathcal{S}}(a_i\gamma_i,a_i'\gamma_i).\]

 By assumption that $\mathcal{C}$ has Condition (Free), 
 there exist $\tau_{{\mathcal{S}}}:W_{\mathcal{S}} \to W'_{\mathcal{S}} \in \mathcal{C}$ and $
 u_{\mathcal{S}}, u'_{\mathcal{S}}\in W_{\mathcal{S}}$ such that $\delta_{\mathcal{S}}(u_{\mathcal{S}} \tau_{\mathcal{S}}, u'_{\mathcal{S}} \tau_{\mathcal{S}})$ is true in $W'_{\mathcal{S}}$. Further, for each $i\in I$, as 
 $\delta_{\mathcal{S}}(a_i\gamma_i,a_i'\gamma_i)$ is true in $A'_i$,
 there exists an $S$-pomorphism $\nu_i:W_{\mathcal{S}}'\rightarrow A_i'$
 such that
  $ u_{\mathcal{S}}\, \tau_{\mathcal{S}} \nu_{i} = a_i \gamma_{i},\, u_{\mathcal{S}}' 
  \tau_{\mathcal{S}} \nu_{i}= a'_i \gamma_{i}$ 
  and 
$ W_{\mathcal{S}} \tau_{\mathcal{S}}\nu_{i} \subseteq A_i \gamma_{i}$.

We have $\delta_{\mathcal{S}}(u _{\mathcal{S}}\tau_{\mathcal{S}},u'_{\mathcal{S}} \tau_{\mathcal{S}})$ is true in $W'_{\mathcal{S}}$ and $\gamma_{\mathcal{S}}(b,b')$ is true in $B$, giving that $u_{\mathcal{S}} \tau_{\mathcal{S}} \otimes b = u'_{\mathcal{S}} \tau_{\mathcal{S}}\otimes b' $ in $W'_{\mathcal{S}} \otimes B$.
As $B$ is a $\mathcal{C}$-flat left $S$-poset
and $\tau_{\mathcal{S}}:W_{\mathcal{S}} \to W'_{\mathcal{S}} \in \mathcal{C}$, we
have that $u_{\mathcal{S}}  \otimes b = u'_{\mathcal{S}} \otimes b' $
in $W_{\mathcal{S}}\otimes B$, say via a double ordered
tossing with double ordered skeleton $\mathcal{U}$. It follows that
\[W_{\mathcal{S}}\models \delta_{\mathcal{U}}(u_{\mathcal{S}},u'_{\mathcal{S}})
\mbox{ and }B\models \gamma_{\mathcal{U}}(b,b').\]
By $(ii)$ of Remark~\ref{rem:formulas}, we have that
\[A_i\gamma_i\models \delta_{\mathcal{U}}(u_{\mathcal{S}}\tau_{\mathcal{S}}\nu_i,u'_{\mathcal{S}}
\tau_{\mathcal{S}}\nu_i),\]
that is,
\[A_i\gamma_i\models \delta_{\mathcal{U}}(a_i\gamma_i,a_i'\gamma_i).\]

Writing $\mathcal{U}=(\mathcal{U}_1,\mathcal{U}_2)$ where
$\mathcal{U}_1$ has length $h$
and  $\mathcal{U}_2$ has length $k$,
we have that there are elements $w_{i,2},\hdots ,w_{i,h},
z_{i,2},\hdots ,z_{i,k}\in A_i$ such that
\[\epsilon_{\mathcal{S}_1}(a_i\gamma_i,w_{i,2}\gamma_i,\hdots ,w_{i,h}\gamma_i,a_i'\gamma_i)
\mbox{ and }
\epsilon_{\mathcal{S}_2}(a'_i\gamma_i,z_{i,2}\gamma_i,\hdots ,z_{i,k}\gamma_i,a_i\gamma_i)\]
are true.
But $\gamma_i$ is an {\em embedding}, so that 
\[\epsilon_{\mathcal{S}_1}(a_i,w_{i,2},\hdots ,w_{i,h},a_i')
\mbox{ and }
\epsilon_{\mathcal{S}_2}(a'_i,z_{i,2},\hdots ,z_{i,k},a_i)\]
hold in $A_i$. Hence $\delta_{\mathcal{U}}(a_i,a_i')$ is true in each $A_i$
and so $\delta_{\mathcal{U}}(\underline{a},\underline{a}')$ holds in $A$.
Together with $\gamma_{\mathcal{U}}(b,b')$ being true in $B$, we deduce
that $\underline{a}\otimes b=\underline{a}'\otimes b'$ in $A\otimes B$,
as required.
\end{proof}

We now come to our first main result. The technique used is inspired by that of \cite{bulmanfleminggould}, but there are some differences: first we are working in a more general context and second, we are dealing with orderings.
 
 \begin{thm}\label{thm:FoutP}
 Let $\mathcal{C}$ be a class of ordered embeddings of right $S$-posets satisfying Condition (Free). Then the following conditions are equivalent for a pomonoid $S$:

 (1) the class $\mathcal{(CF)}$ is axiomatisable;
 
 (2) the class $\mathcal{(CF)}$ is closed under formation of ultraproducts;
 
 (3) for every double ordered skeleton $\mathcal{S} \in \mathbb{DOS}$ there exist finitely many double ordered replacement skeletons $\mathcal{S}_{1},\hdots,\mathcal{S}_{\alpha(\mathcal{S})}$ such that, for any embedding $\gamma: A \to A'$ in $\mathcal{C}$ and any $\mathcal{C}$-flat left $S$-poset $B$, if $(a \gamma ,b),\, (a' \gamma , b') \in A' \times B$  are connected by a double ordered tossing  $\mathcal{T}$ over $A'$ and $B$ with $ \mathcal{S}(\mathcal{T})= \mathcal{S}$, then $(a,b)$ and $(a',b')$ are connected by a double ordered tossing ${\mathcal{T}}^{'}$ over $A$ and $B$ such that $\mathcal{S}({\mathcal{T}}^{'})= \mathcal{S}_{k}$, for some $ k \in \{1,\cdots,\alpha(\mathcal{S})\}$;
 
 (4) for every double ordered skeleton $\mathcal{S} \in \mathbb{DOS}$ there exists finitely many double ordered replacement skeletons $\mathcal{S}_{1},\hdots,\mathcal{S}_{\beta(\mathcal{S})}$ such that, for any $\mathcal{C}$-flat left $S$-poset $B$, if $(u_{\mathcal{S}} \tau_{\mathcal{S}},b)$ and $(u'_{\mathcal{S}} \tau_{\mathcal{S}},b')$ are connected by the double ordered 
  tossing $\mathcal{T}$ over $W'_{\mathcal{S}}$ and $B$ (with $\mathcal{S}(\mathcal{T}) = \mathcal{S}$), then $(u_{\mathcal{S}}, b)$, and $(u'_{\mathcal{S}},b')$ are connected by a double ordered tossing ${\mathcal{T}}^{'}$ over $W_{\mathcal{S}}$ and $B$ such that $\mathcal{S}({\mathcal{T}}^{'})= \mathcal{S}_{k}$, for some $ k \in \{1,\cdots,\beta(\mathcal{S})\}.$
\end{thm}
\begin{proof}
The implication $(1)$ implies $(2)$ is clear from \L os's Theorem.

To prove $(2) \Rightarrow (3) $, we suppose that ${\mathcal {CF}}$,
 the class of $\mathcal{C}$-flat left $S$-posets
 is closed under formation of ultraproducts and that $(3) $ is false. 
Let $J$ be the family of finite subsets of 
$ \mathbb{DOS}$.
We suppose that there exists a double ordered skeleton 
$ \mathcal{S}\in \mathbb{DOS}$
such that for every subset $f$ of $J$, there exists an embedding $\gamma_{f}:A_{f} \to A'_{f} \in \mathcal{C}$, a $\mathcal{C}$-flat left $S$-poset $B_{f}$, and pairs $(a_f \gamma_{f},b_f),\,(a'_{f} \gamma_{f},b'_{f}) \in A'_{f} \times B_{f} $ such that $(a_f \gamma_{f},b_{f})$ and $(a'_{f} \gamma_{f},b'_{f})$ are connected over $A'_{f}$ and $B_{f}$  by a double ordered tossing $\mathcal{T}_{f}$ with double ordered skeleton $\mathcal{S}$, but such that no double ordered replacement tossing over $A_{f}$ and $B_{f}$ connecting $(a_f , b_f)$ and $(a'_f,b'_f)$ has a double ordered skeleton belonging to the set $f$.

Let $J_{\mathcal {S}}=\{f \in J: \mathcal {S}  \in f \}$ for 
each $\mathcal {S}  \in \mathbb {DOS} $. Then there exists an
 ultrafilter $\Phi$ on $J$ containing each 
$J_{\mathcal {S}}$, as
each intersection of finitely many of the sets  $J_{\mathcal {S}}$
 is non-empty.

We now define $A' = \prod_{f \in J}A'_{f},\,A= \prod_{f \in J} A_f $ and $B= \prod_{f \in J} B_f $. Let $\gamma:A \to A'$ be the embedding  given by $(a_f) \gamma = (a_f \gamma_{f}) $.  
We note here that $ \underline{a} \gamma \otimes \underline{b}= 
\underline {a}' \gamma \otimes \underline{b}'$ in $ A{'} \otimes B$, 
where  $\underline{a}=(a_f),\,\underline{a}'=(a'_{f}), \, \underline{b}=(b_f) $ and $\underline{b}'=(b_{f}')$ and that this equality is determined by a double ordered tossing over $A'$ and $B$ (the `` product'' of the double ordered tossings $\mathcal{T}_{f}$'s) having double ordered skeleton $\mathcal{S}$. It follows that the equality  
for $\underline{a} \gamma \otimes
 \underline {b}_{\Phi}= \underline{a}'  \gamma \otimes \underline{b}'_{\Phi}$ 
holds also in $A' \otimes \mathcal{U}$ where $\mathcal{U}= (\prod_{f \in J}B_f)/{\Phi}$,
 and can be determined by a double ordered tossing over $A'$ and $\mathcal{U}$ with 
double ordered skeleton $\mathcal S$.

 By assumption, $\mathcal{U}$ is $\mathcal{C}$-flat, and by Lemma ~\ref{lem:prodC} above, $\underline{a} \otimes \underline{b}_{\Phi }= \underline{a}'\otimes b'_{\Phi}$ in $ A \otimes \mathcal{U}$, say via a double ordered tossing with double ordered
 skeleton $\mathcal{V}=(\mathcal{V}_1,\mathcal{V}_2)$ of length $h+k$, say
 \[\mathcal{V}_1=(d_1,e_1,\hdots ,d_h,e_h)\mbox{ and }
 \mathcal{V}_2=(g_1,\ell_1,\hdots ,g_k,\ell_k).\]
 
 Hence
 \[A\models \delta_{\mathcal{V}}(\underline{a},\underline{a}')
 \mbox{ and }
 \mathcal{U}\models \gamma_{\mathcal{V}}(\underline{b}_{\Phi},\underline{b}'_{\Phi}).\]
 Certainly $A_f\models \delta_{\mathcal{V}}(a_f,a'_f)$ for every $f$. Considering
 now the truth of $\gamma_{\mathcal{V}}(\underline{b}_{\Phi},\underline{b}'_{\Phi})$,
 there exist
 \[(b_{1,f})_{\Phi},\hdots ,(b_{h,f})_{\Phi},(c_{1,f})_{\Phi},\hdots,
 (c_{k,f})_{\Phi}\in\mathcal{U}\]
 such that
 \[\begin{array}{cc}
 \begin{array}{rcl}
 \underline{b}_{\Phi}&\leq&d_1(b_{1,f})_{\Phi}\\
 e_1(b_{1,f})_{\Phi}&\leq &d_2(b_{2,f})_{\Phi}\\
 &\vdots&\\
 e_h(b_{h,f})_{\Phi}&\leq &\underline{b}'_{\Phi}
 \end{array}&
 \begin{array}{rcl}
 
 \underline{b}'_{\Phi}&\leq&g_1(c_{1,f})_{\Phi}\\
 \ell_1(c_{1,f})_{\Phi}&\leq &g_2(c_{2,f})_{\Phi}\\
 &\vdots&\\
 \ell_k(c_{k,f})_{\Phi}&\leq &\underline{b}_{\Phi}.
 \end{array}
 \end{array}\]
 As $\Phi $ is closed under finite intersections, there exists $D \in \Phi$ such that
 
 \[\begin{array}{cc}
 \begin{array}{rcl}
 b_f&\leq&d_1b_{1,f}\\
 e_1b_{1,f}&\leq &d_2b_{2,f}\\
 &\vdots&\\
 e_hb_{h,f}&\leq &b_f'
 \end{array}&
 \begin{array}{rcl}
 
 b_f'&\leq&g_1c_{1,f}\\
 \ell_1c_{1,f}&\leq &g_2c_{2,f}\\
 &\vdots&\\
 \ell_kc_{k,f}&\leq &b_f
 \end{array}
 \end{array}\]    
for all $f \in D$.

Now suppose that $ f \in D \, \cap \, J_{\mathcal V}$, then from the double ordered  tossing just considered, we see that $\mathcal{V}$ is the double ordered skeleton of a double ordered tossing over $A_f$ and $B_f$ connecting the pairs $(a_f,b_f)$ and $(a'_f,b'_f)$; that is, $\mathcal{V}$ a double ordered replacement skeleton for the double ordered skeleton $\mathcal{S}$ of the double ordered  tossing $\mathcal{T}_{f}$. But $\mathcal{V}$ belongs to $f$, a contradiction. This completes the proof that $(2)$ implies that $(3)$.

It is clear that $(3)$ implies that $(4)$.
 
Now we want to prove that $ (4) \Rightarrow  (1) $. We  assume 
that $(4)$ holds. We aim to use this condition to construct a set of axioms for $\mathcal{CF}$.

Let $\mathbb{S}_{1}$ denote the set of all elements of $\mathbb{DOS}$ such that if $\mathcal{S} \in \mathbb{S}_{1}$, then there is no $\mathcal{C}$-flat left $S$-poset $B$ such that $\gamma_{\mathcal{S}}(b,b') \in B$ for any $ b,b' \in B$. For $\mathcal{S} \in \mathbb{S}_{1}$ we put
$$ \psi_{\mathcal{S}}: (\forall x)(\forall x')\neg \gamma_{\mathcal{S}}(x,x')$$

For
 ${\mathcal S} \in {\mathbb S_2}= \mathbb{DOS} \setminus \mathbb{S}_{1}  $, there must be a $B \in \mathcal{CF}$ and $b,b' \in B$  such that $\gamma_{\mathcal{S}}(b,b')$ is true in $B$, whence there is a double ordered  tossing from $(u_{\mathcal{S}} \tau_{\mathcal{S}},b)$ to $(u'_{\mathcal{S}} \tau_{\mathcal{S}},b')$ over $W'_{\mathcal{S}}$ and $B$ with double ordered skeleton $\mathcal{S}$.

Let $\mathcal{S}_{1}, \cdots, \mathcal{S}_{\beta(\mathcal{S})}$ be a minimum set of double ordered replacement skeletons for double ordered  tossings with double ordered skeleton $\mathcal{S}$ connecting pairs of the form $(u_{\mathcal{S}} \tau_{\mathcal{S}},c)$ to $(u'_{\mathcal{S}} \tau_{\mathcal{S}},c')$ where $ c,c' \in C$  and $C$ ranges over $\mathcal{CF}$. Hence for each $ k$ in $ \{ 1, \cdots, \beta(\mathcal{S})\}$,
 there exists a $\mathcal{C}$-flat left $S$-poset $C_k$, elements $c_k,c'_k \in C_k$ such that 
 \[W_{\mathcal{S}}\models \delta_{\mathcal{S}_k}(u_{\mathcal{S}},u'_{\mathcal{S}})
 \mbox{ and }C_k\models \gamma_{\mathcal{S}_k}(c_k,c_k').\]
 
 We define $\phi_{\mathcal{S}}$ to be the sentence 
  
$$ \phi_{\mathcal S}:= (\forall y)(\forall y')
\big(\gamma_{\mathcal S}(y,y') \to {\gamma_{\mathcal S_1}(y,y')} \vee
 \hdots \vee {\gamma_{\mathcal S_{\beta(\mathcal S)}}(y,y')}\big).$$

Let $$ \Sigma_{\mathcal{CF}}= \{\psi_{\mathcal S}: {\mathcal S} \in {\mathbb S_1} \} \cup \{\phi_{\mathcal S}: {\mathcal S} \in {\mathbb S_2}\}.$$

\noindent We claim that $\Sigma_{\mathcal{CF}}$ axiomatises $\mathcal{CF}$.

Suppose first that $D$ is any $\mathcal{C}$-flat left $S$-poset.
By choice of $\mathbb{S}_{1}$, it is clear that $D \models \psi_{\mathcal{S}}$ for any $\mathcal{S} \in \mathbb{S}_{1}$.

Now take any ${\mathcal S} \in {\mathbb S_2}$, and 
suppose that $ d,d' \in D$ are such that $D$ satisfies 
$\gamma_{\mathcal S}(d,d')$. Then, as noted earlier $(u_{\mathcal{S}} \tau_{\mathcal{S}},d)$ and $(u'_{\mathcal{S}} \tau_{\mathcal{S}},d')$ are joined over $W'_{\mathcal{S}}$ and $D$ by a double ordered tossing with double ordered skeleton $\mathcal{S}$, and therefore, by assumption, there is a double ordered tossing over $W_{\mathcal{S}}$ and $D$ joining $(u_{\mathcal{S}},d)$ and $(u'_{\mathcal{S}}, d')$ with double ordered skeleton $\mathcal{S}_{k}$ for some $ k \in \{ 1, \cdots, \beta(\mathcal{S})\}$. It is now clear that $\gamma_{\mathcal{S}_{k}}(d,d')$ holds  in $D$, as required. We have now shown that $D \models \Sigma_{\mathcal{CF}}$.

Finally we show that a left $S$-poset $C$ that satisfies 
$\Sigma_{\mathcal{CF}}$ must be a $\mathcal{C}$-flat. We need  to show that condition $(3) $ of 
Lemma~\ref{lem:coflat}
 holds for $C$. Let 
${\mathcal S}
 \in \mathbb{DOS}$ and suppose we have a double ordered 
tossing with double ordered skeleton $\mathcal{S}$
connecting $(u_{\mathcal{S}}\tau_{\mathcal{S}},c)$ and
$(u_{\mathcal{S}}'\tau_{\mathcal{S}},c')$ over $W_{\mathcal{S}}'$ and $C$. Then
\[W_{\mathcal{S}}'\models \delta_{\mathcal{S}}(u_{\mathcal{S}}\tau_{\mathcal{S}},
u_{\mathcal{S}}'\tau_{\mathcal{S}})
\mbox{ and }
C\models \gamma_{\mathcal{S}}(c,c').\]

If ${\mathcal S}$ belonged to ${\mathbb{S}_1}$, then $C$ would satisfy the sentence $(\forall y)(\forall y'){\neg \gamma_{\mathcal S}(y,y')}$ and so $\neg \gamma_{\mathcal S}(c,c')$ would hold, which is would be a contradiction. Therefore we conclude that $ {\mathcal S}$ belongs to ${\mathbb S_2}$.
Because $C$ satisfies $\phi_{\mathcal S}$ and 
because $\gamma_{\mathcal S}(c,c')$ holds, it follows that 
$ \gamma_{\mathcal S_{k}}(c,c')$ holds for some
 $k \in \{1,2,\cdots, \beta (\mathcal S) \}$. But
 $W_{\mathcal{S}}\models \delta_{\mathcal{S}_k}(u_{\mathcal{S}},u'_{\mathcal{S}_k})$, 
 whence $(u_{\mathcal{S}},c) $ and $(u'_{\mathcal{S}},c')$ are connected
 via a double ordered tossing 
 over $W_{\mathcal{S}}$ and $C$ with double ordered skeleton $\mathcal{S}_k$, showing that $C$ is $\mathcal{C}$-flat.  
\end{proof}

We recall that the definition of a flat $S$-poset is that it is $\mathcal{C}$-flat where
$\mathcal{C}$ is the class of {\em all}  embeddings of right $S$-posets.
The class of  all flat left $S$-posets is denoted by $\mathcal{F}$.

By Lemma~\ref{lem:allfree}, the class of all right $S$-posets has Condition (Free),  so from
Theorem~\ref{thm:FoutP}, we immediately have the following corollary. 


\begin{cor}\label{cor:flat}  The following conditions are equivalent for an ordered  monoid $S$:

 (1) the class $\mathcal{F}$ is axiomatisable;
 
 (2) the class $\mathcal{F}$ is closed under formation of ultraproducts;
 
  (3) for every double ordered skeleton $\mathcal{S} \in \mathbb{DOS}$ there exist finitely many double ordered replacement skeletons $\mathcal{S}_{1},\cdots,\mathcal{S}_{\alpha(\mathcal{S})}$ such that, for any right $S$-poset ordered  
  embedding $\gamma:A\rightarrow A'$, and any flat left $S$-poset $B$, if $(a\gamma  ,b),\, (a'\gamma , b') \in A' \times B$  are connected by a double ordered  tossing  $\mathcal{T}$ over $A'$ and $B$ with $ \mathcal{S}(\mathcal{T})= \mathcal{S}$, then $(a,b)$ and $(a',b')$ are connected by a double ordered tossing ${\mathcal{T}}^{'}$ over $A$ and $B$ such that $\mathcal{S}({\mathcal{T}}^{'})= \mathcal{S}_{k}$, for some $ k \in \{1,\cdots,\alpha(\mathcal{S})\}$;
 
 (4) for every double ordered skeleton $\mathcal{S} \in \mathbb{DOS}$ there exist finitely many double ordered replacement skeletons $\mathcal{S}_{1},\cdots,\mathcal{S}_{\alpha(\mathcal{S})}$ such that, for any right $S$-poset $A$ and any flat left $S$-poset $B$, if $(a  ,b),\, (a' , b') \in A \times B$  are connected by a double ordered tossing  $\mathcal{T}$ over $A$ and $B$ with $ \mathcal{S}(\mathcal{T})= \mathcal{S}$, then $(a,b)$ and $(a',b')$ are connected by a double ordered tossing ${\mathcal{T}}^{'}$ over $aS \cup a'S$ and $B$ such that $\mathcal{S}({\mathcal{T}}^{'})= \mathcal{S}_{k}$, for some $ k \in \{1,\cdots,\alpha(\mathcal{S})\}$;
 
 (5) for every double ordered skeleton $\mathcal{S} \in \mathbb{DOS}$ there exists finitely many double ordered replacement skeletons $\mathcal{S}_{1},\cdots,\mathcal{S}_{\beta(\mathcal{S})}$ such that, for any flat left $S$-poset $B$, if $([x],b)$ and $([x'],b') $ are connected by a double ordered  tossing $\mathcal{T}$ over $F^{m+n}/\equiv_{\mathcal{S}}$ and $B$ with $\mathcal{S}(\mathcal{T}) = \mathcal{S}$, then $([x], b)$, and $([x'],b')$ are connected by a double ordered  tossing ${\mathcal{T}}^{'}$ over $[x]S \cup [x']S$ and $B$ such that $\mathcal{S}({\mathcal{T}}^{'})= \mathcal{S}_{k}$, for some $ k \in \{1,\cdots,\beta(\mathcal{S})\}.$
 
\end{cor}

\subsection{Axiomatisability of $\mathcal{CF}$ in the general case}\label{subsec:sposetsoutFP}

We continue to consider a class $\mathcal{C}$ of  ordered embeddings of right $S$-posets , but now drop our assumption that Condition (Free)  holds. The results and proofs of this section are analogous to those for weakly flat 
$S$-acts in \cite{bulmanfleminggould}. Note that the conditions in (3) below appear weaker than those in 
Theorem~\ref{thm:FoutP}, as we are only asking that for specific elements $a,a'$ and double ordered  skeleton
$\mathcal{S}$, there are finitely many double ordered replacement skeletons, in the sense made specific below.

\begin{thm}\label{thm:sposetsoutFP} Let $\mathcal{C}$ be a class of   embeddings of right $S$-posets.

The following conditions are equivalent:

(1) the class $\mathcal{CF}$ is axiomatisable;

(2) the class $\mathcal{CF}$ is closed under ultraproducts;

(3) for every double ordered skeleton $\mathcal{S}\in\mathbb{DOS} $  and $a,a' \in A$, where $\mu:A \to A'$ is in $\mathcal{C}$, there exist finitely many double ordered skeleton $\mathcal{S}_{1},\cdots,\mathcal{S}_{\alpha(a,\mathcal{S},a',\mu)}$, such that for any $\mathcal{C}$-flat left $S$-poset $B$, if $(a \mu , b),\,(a' \mu ,b')$ are connected by a double ordered tossing $\mathcal{T}$ over $A'$ and $B$ with $\mathcal{S}(\mathcal{T})= \mathcal{S}$, then $(a,b)$ and $(a',b')$ are connected by a double ordered tossing $\mathcal{T}'$ over $A$ and $B$ such that $\mathcal{S}(\mathcal{T}')=\mathcal{S}_{k}$, for some $ k \in \{1,\cdots,\alpha(a , \mathcal{S},a',\mu)\}$.
\end{thm}
\begin{proof}
The implication $(1)$ implies $(2)$ is clear from \L os's Theorem.

To prove $(2) \Rightarrow (3) $, we suppose that $\mathcal {CF}$,
 the class of $\mathcal{C}$-flat left $S$-posets,
 is closed under formation of ultraproducts, and assume that $(3) $ is false. 
Let $J$ be the family of finite subsets of 
$ \mathbb{DOS}$.
We suppose that for some double ordered skeleton 
$ {\mathcal S}\in \mathbb{DOS}$, for some ordered embedding  
 $\mu:A \to A' \in \mathcal{C}$, and for some $a,a' \in A$, for every $ f \in J$ there is a $\mathcal{C}$-flat left $S$-poset $B_{f}$, and $b_f , b'_f \in B_f$ such that $(a \mu,b_{f})$ and $(a'\mu,b'_{f})$ are connected over $A'$ and $B_{f}$  by a double ordered tossing $\mathcal{T}_{f}$ with double ordered skeleton $\mathcal{S}$, but such that no double ordered replacement tossing over $A$ and $B_{f}$  connecting $(a , b_f)$ and $(a',b'_f)$ has a double ordered skeleton belonging to the set $f$.

Let $J_{\mathcal S}=\{f \in J: {\mathcal S } \in f \}$ for 
each ${\mathcal S } \in {\mathbb S} $. Now we are able to define an
 ultrafilter $\Phi$ on $J$ containing each 
$J_{\mathcal S}$ for all $ {\mathcal S} \in \mathbb S$, as
each intersection of finitely many of the sets  $J_{\mathcal S}$
 is non-empty.

We note here that ${a} \mu \otimes \underline{b}= 
{a}^{'} \mu \otimes \underline{b}^{'}$ in $ A{'} \otimes B$,  where $B= \prod_{f \in J}B_{f}$, $\underline{b}=(b_f)$ and $\underline{b}'=(b_f')$, and that this equality is determined by a double ordered tossing over $A'$ and $B$ (the ``product'' of the double ordered tossings $\mathcal{T}_{f}$) having double ordered  skeleton $\mathcal{S}$. It follows that the equality  
for ${a} \mu \otimes
 \underline {b}_{\Phi}= {a}^{'} \mu \otimes \underline{b}^{'}_{\Phi}$ 
holds also in $A' \otimes {\mathcal{U}}$ where $\mathcal{U}= (\prod_{f \in J}B_f)/{\Phi}$,
 and can be determined by a double ordered tossing over $A'$ and $\mathcal{U}$ with double ordered 
skeleton $\mathcal S$.

By assumption $\mathcal{U}$ is $\mathcal{C}$-flat, so that $(a,\underline{b}_{\Phi})$ and $(a',\underline{b}^{'}_{\Phi})$ are connected via a double ordered  replacement tossing over $A$ and $\mathcal{U}$, with double ordered
skeleton $\mathcal{V}$ say. Hence
\[A\models \delta_{\mathcal{V}}(a,a')\mbox{ and }
\mathcal{U}\models \gamma_{\mathcal{V}}(\underline{b}_{\Phi},\underline{b}_{\Phi}).\]

As in Theorem~\ref{thm:FoutP}, there exists $D\in \Phi$ such that $B_f\models \gamma_{\mathcal{V}}(b_f,b_f')$ for all $f\in D$.

Now suppose that $ f \in D \cap J_{\mathcal V}$. Then $\mathcal{V}$ is the double  ordered skeleton of a double ordered tossing over $A$ and $B_f$ connecting the pairs $(a,b_f)$ and $(a',b'_f)$; that is, $\mathcal{V}$ is a double ordered  replacement skeleton for double ordered skeleton $\mathcal{S}$ of the  double ordered tossing $\mathcal{T}_{f}$. But $\mathcal{S}$ belongs to $f$, a contradiction. This completes the proof that $(2)$ implies that $(3)$.

Finally, suppose that $(3)$ holds. Let
\[\mathbb{T}'=\{ (a,\mathcal{S},a',\mu):\mathcal{S}\in \mathbb{DOS}, \mu:A\rightarrow A'\in \mathcal{C},
 a,a'\in A', \delta_{\mathcal{S}}(a\mu,a'\mu)\mbox{ holds}\}.\]

We introduce sentences corresponding to  elements of $\mathbb{T}'$ in such a way that the resulting set of sentences axiomatises the class $\mathcal{CF}$.

We let $\mathbb{T}_{1}$ be the set of 
$(a,\mathcal{S},a',\mu)\in\mathbb{T}'$ 
such that $\gamma_{\mathcal{S}}(b,b')$ does not hold for any $b,b'$ in
any $\mathcal{C}$-flat left $S$-poset $B$, and put $\mathbb{T}_{2}= \mathbb{T}' \setminus \mathbb{T}_{1}$. For $ T= ( a,\mathcal{S},a',\mu) \in \mathbb{T}_{1}$ we let $$ \psi_{T}= \psi_{\mathcal{S}}:(\forall x)(\forall x')\neg\gamma_{\mathcal{S}}(x,x').$$ If $T=(a,\mathcal{S},a',\mu) \in \mathbb{T}_{2}$, then $\mathcal{S}$ is the double ordered skeleton of some double ordered tossing joining $(a \mu ,b)$ to $(a' \mu,b')$ over $A'$ and some $\mathcal{C}$-flat left $S$-poset $B$. By our assumption $(3)$, there is a finite list of double ordered replacement skeletons $\mathcal{S}_{1},\cdots,\mathcal{S}_{\alpha(T)}.$ Choosing 
$\alpha(T)$ to be minimal,  for each $k \in \{1,\cdots,\alpha(T)\}$,
 there exist a $\mathcal{C}$-flat left $S$-poset $C_k$ and
 elements $c_k,c_k' \in C_k$, such that 
 \[ A\models \delta_{\mathcal{S}_k}(a,a')\mbox{ and }C_k\models \gamma_{\mathcal{S}_k}(c_k,c_k').\]
 We let $\phi_{T}$ be the sentence 
$$\phi_{T}=(\forall y)(\forall y')(\gamma_{\mathcal{S}}(y,y') \to \gamma_{\mathcal{S}_{1}}(y,y')\vee \cdots \vee \gamma_{\mathcal{S}_{\alpha(T)}}(y,y'))$$
Let $$\sum_{\mathcal{CF}}=\{ \psi_{T}:T \in \mathbb{T}_{1} \} \cup \{\phi_{T}:T \in \mathbb{T}_{2} \}$$
We claim that $\sum_{\mathcal{CF}}$ axiomatises $\mathcal{CF}$.

Suppose first that $D$ is any $\mathcal{C}$-flat left $S$-poset. Let $ T 
=(a,\mathcal{S},a',\mu)\in \mathbb{T}_{1}$. Then
$\gamma_{\mathcal{S}}(b,b')$ is not true for any $b,b'\in B$, for any
$\mathcal{C}$-flat left $S$-poset $B$, so certainly $D \models \psi_{T}$. 

On the other hand, let $T=(a,\mathcal{S},a',\mu) \in \mathbb{T}_{2}$, 
and let $d,d' \in D$ be such that $\gamma_{\mathcal{S}}(d,d')$ is true.
Together with the fact $\delta_{\mathcal{S}}(a\mu,a'\mu)$
holds,  we have  that $(a \mu, d)$ is connected to $(a' \mu,d')$ over $A'$ and $D$ via  a  
double ordered tossing with double ordered skeleton $\mathcal{S}$. Because $D$ is $\mathcal{C}$-flat, $(a,d)$ and $(a',d')$ are connected over $A$ and $D$, and by assumption $(3)$, we can take the double ordered 
replacement tossing to have double ordered skeleton one of $\mathcal{S}_{1},\cdots, \mathcal{S}_{\alpha(T)}$, say $\mathcal{S}_{k}$. Thus $D \models \gamma_{\mathcal{S}_{k}}(d,d^{'})$ and it follows that $D \models \phi_{T}$. Hence $D$ is a model of $\sum_{\mathcal{CF}}$.

Conversely, we show that every model of $\sum_{\mathcal{CF}}$ is $\mathcal{C}$-flat. Let ${C} \models \sum_{\mathcal{CF}}$ and suppose that $ \mu:A \to A' \in \mathcal{C}, \, a,a' \in A, \, c,c' \in C$ and 
$a\mu\otimes c=a'\mu\otimes c'$ in $A'\otimes C$, say with double ordered tossing having
double ordered skeleton $\mathcal{S}$. Then the
quadruple $T = (a,\mathcal{S},a',\mu)\in\mathbb{T}'$. Since $\gamma_{\mathcal{S}}(c,c')$ holds, $C$ cannot be a model of $\psi_{T}$. Since $C \models \sum_{\mathcal{CF}}$ it follows that $ T \in \mathbb{T}_{2}$. But then $\phi_{T}$ holds in $C$ so that for some $ k \in \{1,\cdots,\alpha(T)\}$ we have that $\gamma_{\mathcal{S}_{k}}(c,c')$ is true.
We also know that $A\models \delta_{\mathcal{S}_k}(a,a')$, so that we have double ordered  tossing over $A$ and $C$ connecting $(a,c)$ to $(a',c')$. Thus $C$ is $\mathcal{C}$-flat.

\end{proof}

\bigskip

We now apply Theorem~\ref{thm:sposetsoutFP} to the class of all embeddings
of right ideals into $S$, and the class of all embeddings of principal right ideals
into $S$. 
In these corollaries we do not need to mention the embeddings $\mu$, since they are all inclusion maps of right ideals into $S$.

\begin{cor}\label{cor:axweaklyflat}
The following are equivalent for a pomonoid $S$:

(i) the class $\mathcal{WF}$ is axiomatisable;

(ii) the class $ \mathcal{WF}$ is closed under ultraproducts;

(iii) for every double ordered  skeleton $\mathcal{S}$  and $a,a' \in S$ there exists finitely many double ordered skeletons  $\mathcal{S}_1,\cdots, \mathcal {S}_{\beta(a,\mathcal{S},a')} $  such that for any weakly flat left $S$-poset $B$, if $(a,b)$, $(a',b') \in S \times B$ are connected by a double ordered tossing $\mathcal{ T}$ over $S$ and $B$ with $ \mathcal{S(T)}= \mathcal{S}$ then $(a,b)$ and $(a',b')$ are connected by a double ordered tossing $\mathcal{T'}$ over $aS \cup a'S$ and $B$ such that $ \mathcal{S(T')}=\mathcal{S}_k$ for some $k \in \{1,\cdots,\beta(a,\mathcal{S},a' )\}$.
\end{cor}

We end this section by considering the axiomatisability of principally weakly flat $S$-posets. We first remark that if $aS$ is a principal right ideals of $S$ and $B$ is a left $S$-poset, then 

$$ a u \otimes b = av \otimes b' \,\, {\mbox {in} }  \,\,aS \otimes B  \,\,{\mbox {if and only if} } \,\, a  \otimes ub = a \otimes vb' \,\, {\mbox  {in}} \,\, aS \otimes B$$
with a similar statement for $S \otimes B$. Thus $B$ is principally weakly flat if and only if for all $ a \in S$, if $ a \otimes b = a \otimes b' $ in $S \otimes B$, then $ a \otimes b = a \otimes b'$ in $aS \cup B.$ From Theorem~\ref{thm:sposetsoutFP} and its
proof we have the following result for $ \mathcal{PWF}$.

\begin{cor}\label{cor:axprweaklyflats}
The following conditions are equivalent for a pomonoid $\mathcal{S}$:

(i) the class $ \mathcal{PWF}$ is axiomatisable;

(ii) the class $\mathcal{PWF}$ is closed under ultraproducts;

(iii) for every double ordered skeleton $\mathcal{S}$ over $S$ and $ a \in S$ there exists finitely many double ordered skeletons $ \mathcal{S}_{1},\cdots, \mathcal{S}_{\gamma( a, \mathcal{S})}$ over $S$, such that for any principally weakly flat left $S$-poset $B$, if $(a,b), \, (a,b') \in S \otimes B$ are connected by a double ordered tossing $\mathcal{T}$ over $S$ and $B$ with $\mathcal{S(T)=S}$ then $(a,b)$ and $(a,b')$ are connected by a double ordered tossing $\mathcal{T}'$ over $aS $ and $B$ such that $\mathcal{S(\mathcal{T}')}= \mathcal{S}_{k}$ for some $k \in \{ 1,\cdots,\gamma(a, \mathcal{S}) \}$.
\end{cor}

\section{Axiomatisability of $\mathcal{CPF}$}\label{sec:axiopflatsposet}

In this section we briefly  explain how the methods and results of
Section~\ref{sec:tossings} may be adapted to the case when $-\otimes B$ preserves  embeddings, rather than merely taking embeddings to monomorphisms. We omit proofs, as
they follow now established patterns. Further details may be found in \cite{shaheen:2010}.

We introduce a condition on a class $\mathcal{C}$ of embeddings
of right $S$-posets called Condition (Free)$^{\leq}$.

Let
\[\mathcal{S}=(s_1,t_1,\hdots, s_m,t_m)\]
be an ordered skeleton of length $m$. We put
\[\delta^{\leq}_{\mathcal{S}}(x,x'):= (\exists x_2\hdots \exists x_m)\epsilon_{\mathcal{S}}
(x,x_2,\hdots ,x_m,x')\]
and
\[\gamma^{\leq}_{\mathcal{S}}(x,x'):=(\exists x_1\hdots \exists x_m)\theta_{\mathcal{S}}
(x,x_1,\hdots ,x_m,x')\]
where $\epsilon$ and $\theta$ are defined as in Section~\ref{sec:tossings}. Notice that similar comments to those in Remark~\ref{rem:formulas} hold, in particular, if $A$ is a right and
$B$ a left $S$-poset, then the pair $(a,b)\in A\times B$ is connected to the
pair $(a',b')\in A\times B$ via an ordered tossing with ordered skeleton $\mathcal{S}$ if and only  if $\delta_{\mathcal{S}}^{\leq}(a,a')$ is true in $A$ and
$\gamma_{\mathcal{S}}^{\leq}(b,b')$ is true in $B$.

\begin{defi}\label{defi:orderedfree} We say that $\mathcal{C}$ satisfies Condition (Free)$^{\leq}$ if for each ordered skeleton $\mathcal{S}$ there is an  embedding $\kappa_{\mathcal{S}}:V_{\mathcal{S}} \to V^{'}_{\mathcal{S}}$ in $\mathcal{C}$ and $v_{\mathcal{S}},\,v^{'}_{\mathcal{S}} \in V_{\mathcal{S}}$ such that $\delta^{\leq}_{\mathcal{S}}(v_{\mathcal{S}}\kappa_{\mathcal{S}},v^{'}_{\mathcal{S}}\kappa_{\mathcal{S}})$ is true in $V^{'}_{\mathcal{S}}$ and further for any  embedding $\mu:A \to A^{'} \in \mathcal{C}$  and any $a,a' \in A$  such that $\delta^{\leq}_{\mathcal{S}}(a \mu, a' \mu )$ is true in $A'$ there is a morphism $\nu:V^{'}_{\mathcal{S}} \to A'$ such that $u_{\mathcal{S}}\kappa_{\mathcal{S}}\nu = a \mu,\,u^{'}_{\mathcal{S}}\kappa_{\mathcal{S}}\nu = a' \mu $ and $V_{\mathcal{S}} \kappa_{\mathcal{S}} \nu \subseteq A \mu.$

\end{defi}

As in Lemma ~\ref{lem:coflat}, we can show that if  $\mathcal{C}$ be a class of  embeddings of right $S$-posets satisfying Condition (Free)$^{\leq}$, then to show that
a left  $S$-poset $B$ is in $\mathcal{CPF}$, that is, $B$ is $\mathcal{C}$-poflat, it is 
enough to show that for any ordered skeleton $\mathcal{S}$,
if $(v_{\mathcal{S}}\kappa_{\mathcal{S}},b)$ and
$(v'_{\mathcal{S}}\kappa_{\mathcal{S}},b')$ are connected by an
ordered tossing over $V'_{\mathcal{S}}$ and
$B$ with ordered skeleton $\mathcal{S}$, then
$(v_{\mathcal{S}},b)$ and $(v'_{\mathcal{S}},b')$ are connected by 
an ordered tossing over $V_{\mathcal{S}}$ and $B$. Moreover, if $B\in \mathcal{CPF}$,
then $B\in \overline{\mathcal{C}}{\mathcal{PF}}$. 

Everything is then in place to prove the next result.

\begin{thm}\label{thm:poflatwithf} Let $\mathcal{C}$ be a class of embeddings of right $S$-posets satisfying Condition (Free)$^{\leq}$. Then the following conditions are equivalent for a pomonoid $S$:

(i) the class   $\mathcal{CPF}$     is axiomatisable;

(ii)  the class $\mathcal{CPF}$ is closed under formation of ultraproducts;
 
 (iii) for every ordered skeleton $\mathcal{S} $ there exist finitely many replacement ordered skeletons $\mathcal{S}_{1},\cdots,\mathcal{S}_{\alpha(\mathcal{S})}$ such that, for any embedding $\gamma: A \to A'$ in $\mathcal{C}$ and any $\mathcal{C}$-poflat left $S$-poset $B$, if $a \gamma \, \otimes \,b\, \leq \, a' \gamma  \, \otimes \, b' \in A' \otimes B$  by an ordered tossing  $\mathcal{T}$ with $ \mathcal{S}(\mathcal{T})= \mathcal{S}$, then $a \otimes b \leq a' \otimes b'$  by an ordered tossing ${\mathcal{T}}^{'}$ over $A$ and $B$ such that $\mathcal{S}({\mathcal{T}}^{'})= \mathcal{S}_{k}$, for some $ k \in \{1,\cdots,\alpha(\mathcal{S})\}$;
 
 (iv) for every ordered skeleton $\mathcal{S}$ there exists finitely many replacement ordered skeletons $\mathcal{S}_{1},\cdots,\mathcal{S}_{\beta(\mathcal{S})}$ such that, for any $\mathcal{C}$-poflat left $S$-poset $B$, if $(v_{\mathcal{S}} \kappa_{\mathcal{S}},b)$ and $(v'_{\mathcal{S}} \kappa_{\mathcal{S}},b')$ are such that $v_{\mathcal{S}} \kappa_{\mathcal{S}}\, \otimes \, b  \leq v'_{\mathcal{S}} \kappa_{\mathcal{S}}\, \otimes \, b'$  by an ordered tossing $\mathcal{T}$ over $V'_{\mathcal{S}}$ and $B$ with $\mathcal{S}(\mathcal{T}) = \mathcal{S}$, then $v_{\mathcal{S}} \, \otimes \, b \leq v'_{\mathcal{S}}\, \otimes \,b'$ are connected by an ordered tossing ${\mathcal{T}}^{'}$ over $V_{\mathcal{S}}$ and $B$ such that $\mathcal{S}({\mathcal{T}}^{'})= \mathcal{S}_{k}$, for some $ k \in \{1,\cdots,\beta(\mathcal{S})\}.$
\end{thm}

To show that the class of all embeddings of right $S$-posets has Condition
(Free)$^{\leq}$, for an ordered skeleton 
\[\mathcal{S}=(s_1,t_1,\hdots, s_m,t_m)\]
we let $F^{m}$ be the free right $S$-poset
 $$ xS \,\dot{\cup} \,x_2 S \,\dot{\cup} \,\hdots  x_m S\,\dot{\cup} \, x' S$$
 and put
 \[T_{\mathcal{S}}= \big \{(xs_1,x_2t_1),(x_2 s_2, x_3 t_2), \hdots,(x_m s_m , x' t_m)\big \}.\]
 Let $=_{\mathcal{S}}$ be $\equiv_{T_{\mathcal{S}}} $, the $S$-poset congruence  which is induced by $T_{\mathcal{S}}$. Abbreviate the order
 $\preceq_{\mathcal{T_{\mathcal{S}}}}$ by
 $\leq_{\mathcal{S}}$ so that $[a]\leq_{\mathcal{S}}[b]$
 for all $(a,b)\in T_{\mathcal{S}}$. We defined an
 {\em ordered standard tossing} from $([x],b)$ to
 $([x'],b')$ where $b,b'\in B$ for a left $S$-poset
 $B$ in the analogous way to a double ordered standard tossing.

The proof of the next lemma follows that of Lemma~\ref{lem:fpflat}.

\begin{lem}\label{lem:fppoflat} The following conditions are  equivalent for a left $S$-poset $B$:

(i) $B$ is po-flat;

(ii) $-\otimes B$ maps the embeddings of $[x]S \cup [x']S$ into $F^{m}/ =_{\mathcal S}$ in the category {\bf Pos-}${\mathbf S}$ to embeddings in the category of ${\bf Pos}$, for every ordered skeleton $\mathcal S$;

(iii)  if the inequality $[x] \otimes b \leq [x'] \otimes b' $ holds by an ordered standard tossing over $F^{m}/ =_{\mathcal S}$ and $B$ with ordered skeleton  $\mathcal S$, then $[x] \otimes b  \leq [x'] \otimes b'$ holds  by an ordered tossing over $[x]S \cup [x']S$ and $B$.
\end{lem}

As in Lemma~\ref{lem:allfree} we then have:

\begin{lem}\label{lem:allpfree} The class Pos-$S$ of all right $S$-posets
has Condition (Free)$^{\leq}$.
\end{lem}
 
We can now deduce the following corollary, which appears without proof
in
\cite{step:2009}. The reader should note that in that article, (weakly)
po-flat $S$-posets are referred to as being (weakly) flat.

\begin{cor}\label{cor:pf}\cite{step:2009}
The following conditions are equivalent for a pomonoid $S$:

(i) the class $\mathcal{PF}$ is axiomatisable;

(ii)the class $\mathcal{PF}$ is closed under formation of ultraproducts;

(iii) for every ordered skeleton $\mathcal{S}$  there exist finitely 
many replacement ordered skeletons $\mathcal{S}_{1}, \hdots,\mathcal{S}_{\alpha(\mathcal{S})} $ such that, for any right $S$-poset $A$ and any poflat left $S$-poset $B$, if $ a \otimes b \leq a' \otimes b' $ exists in $ A \otimes B$ by a ordered tossing $\mathcal{T}$ with ordered skeleton $\mathcal{S}$, then $a \otimes b \leq a' \otimes b' $ also exists in $ (aS \cup a'S) \otimes B$ by a replacement ordered tossing $\mathcal{T}'$ such that $ \mathcal{S }(\mathcal{T}') = \mathcal{S}_{k}$, for some $ k \in \{1, \cdots, \alpha(\mathcal{S})\}.$
\end{cor}

We now drop our assumption that Condition (Free)$^{\leq}$  holds. The proof of the next result follows that of Theorem~\ref{thm:sposetsoutFP}.

\begin{thm}\label{thm:sactsoutFP} Let $\mathcal{C}$ be a class of embeddings of
right $S$-posets over a pomonoid $S$. Then
the following conditions are equivalent for a pomonoid S:

(1) the class $\mathcal{CPF}$ is axiomatisable;

(2) the class $\mathcal{CPF}$ is closed under ultraproducts;

(3) for every ordered skeleton $\mathcal{S}$ over $S$ and $a,a' \in A$, where $\mu:A \to A'$ is in $\mathcal{C}$, there exist finitely many ordered skeletons $\mathcal{S}_{1},\cdots,\mathcal{S}_{\alpha(a,\mathcal{S},a')}$, such that for any $\mathcal{C}$-poflat left $S$-act $B$, if $a \mu  \otimes  b \leq a' \mu  \otimes b'$ by an ordered tossing $\mathcal{T}$ over $A'$ and $B$ with $\mathcal{S}(\mathcal{T})= \mathcal{S}$, then $a \otimes b \leq a' \otimes b'$ by an ordered  tossing $\mathcal{T}'$ over $A$ and $B$ such that $\mathcal{S}(\mathcal{T}')=\mathcal{S}_{k}$, for some $ k \in \{1,\cdots,\alpha(a , \mathcal{S},a')\}$.
\end{thm}

Theorem~\ref{thm:sactsoutFP} can be specialised to the cases where 
$\mathcal{C}$ consists of all inclusions
of (principal) right ideals of $S$ into $S$, 
thus giving necessary and sufficient conditions
on $S$ such that $\mathcal{WPF}$ (a result also found
in \cite{step:2009}) ($\mathcal{PWPF}$) 
is axiomatisable. The statements
of these results are obtained from those of 
Corollaries~\ref{cor:axweaklyflat}
and
\ref{cor:axprweaklyflats}, 
with the word `double' omitted and `flat' replaced by
`poflat'. Further details may be found in \cite{shaheen:2010}.

\section{Axiomatisability of some specific classes of $S$-posets}\label{sec:specific}

We now concentrate on axiomatisability problems for certain
classes  of $S$-posets, in the cases that we can avoid the
`replacement tossings' arguments of the Sections~\ref{sec:tossings}
and \ref{sec:axiopflatsposet}. We consider the classes of $S$-posets satisfying Condition (P) and
(E) (which together give us the class of strongly flat $S$-posets), and the classes of $S$-posets satisfying
Condition
 (EP), (W), (P$_W$), (PWP) and (PWP$_{W}$).

Let $S$ be a pomonoid and let $(s,t)\in S\times S$. We define
\[R^{\leq}(s,t)=\{ (u,v)\in S\times S:su\leq tv\}\mbox{ and }
r^{\leq}(s,t)=\{ u\in S:su\leq tu\}\]
so that $R^{\leq}(s,t)$ is either empty or is an $S$-subposet of
the right $S$-poset
$S\times S$, and $r^{\leq}(s,t)$ is either empty or is a right ideal of $S$.
 Note that in
\cite{step:2009}, $R^{\leq}(s,t)$ and $r^{\leq}(s,t)$ are written as
$R^{<}(s,t)$ and $r^{<}(s,t)$.

\subsection{Conditions (P) and (E) and the class $\mathcal{SF}$}
\label{subsec:pesf}

For completeness we give the following results from 
\cite{step:2009}; they may also be
found in the thesis of the second author \cite{shaheen:2010}. 
The proofs follow closely those of the unordered case in 
\cite{gould},\cite{gould:tartu}
and \cite{gouldpalyutin}.

\begin{thm}\label{pande}\cite{step:2009} Let $S$ be a pomonoid. 

(1) The class of left $S$-posets satisfying Condition (P) is axiomatisable if
and only if for every $s,t \in S$, $R^{\leq}(s,t)$ is empty or is finitely generated.

(2) The class of left $S$-posets satisfying Condition (E) is axiomatisable if
and only if for every $s,t \in S$, $r^{\leq}(s,t)$ is empty or is finitely generated.

(3) The class of $\mathcal{SF}$ of strongly flat left $S$-posets is axiomatisable if
and only if for every $s,t \in S$, $R^{\leq}(s,t)$ is empty or is finitely generated
and $r^{\leq}(s,t)$ is empty or is finitely generated.
\end{thm}

\subsection{Condition (EP)}
\label{subsec:ep}

We recall from Section~\ref{sec:intro} that, in the terminology introduced
above,  a left $S$-poset $A$ satisfies
Condition (EP) if, given $sa\leq ta$ for any $s,t\in S$ and
$a\in A$, we have that
\[a=ua'=va'\mbox{ for some }(u,v)\in R^{\leq}(s,t)\mbox{ and }a'\in A.\]

\begin{thm}\label{thm:(EP)} 
The following conditions are equivalent for a  pomonoid $S$:

(1) the class $\mathcal{EP}$ is axiomatisable;

(2) the class $\mathcal{EP}$ is closed under ultraproducts;

(3)  for any $s,t \in S$ either $sa \not \leq ta$ for all $ a \in A \in \mathcal{EP}$ or there exists a finite subset $f$ of $ \mathbf{R}^{\leq}(s,t)$, such that 
for any $a \in A \in \mathcal{EP}$
\[ sa  \leq  ta \Rightarrow (a,a)=(u,v)b\mbox{ for some }(u,v) \in f,b \in A.\]
\end{thm}

\begin{proof}

$(1) \Rightarrow (2)$ This follows from \L os's Theorem.

$(2) \Rightarrow (3)$ Suppose $sa  \leq  ta$ for some $a \in A \in \mathcal{EP}$ and for each finite subset $f $ of $ \mathbf{R}^{\leq}(s,t)$, there exists $A_{f} \in \mathcal{EP^{\leq}}$, $a_{f} \in A_{f}$ with  $s a_{f} \leq  t a_{f}$ and $(a_{f},a_{f}) \not \in \,f A_{f}.$

Let $J$ be the set of finite subsets of $\mathbf{R}^{\leq}(s,t)$.  For each $(u,v) \in \mathbf{R}^{\leq}(s,t)$ we define $$J_{(u,v)}= \{f \in J:(u,v) \in f\}$$

As  each intersection of finitely many of the sets $J_{(u,v)}$ is non-empty, we are able to define an ultrafilter $\Phi$ on $J$,  such that each $J_{(u,v)} \in \Phi$ for all $(u,v) \in \mathbf{R}^{\leq}(s,t).$

Now $s (a_f) \leq  t (a_f)$ in $A$ where $A= \prod_{f \in J}A_{f}$, and it follows that  the inequality $s (a_f)_{\Phi} \leq  t (a_f){\Phi}$  holds in $\mathcal{U}$ where  $\mathcal{U} = \prod_{f \in J} A_{f} /\Phi$. By assumption  $\mathcal{U}$ lies in $\mathcal{EP}$,  so there exists $(u , v) \in R^{\leq}(s,t)$, and $r_{f} \in A_f$ such that  $$(a_f)_{\Phi}  = u (r_f)_{\Phi} = v  ({r_{f}})_{\Phi}.$$

As $\Phi$ is closed under finite intersections, there must exists $T \in \Phi$ such that 
 $a_{f} = u r_{f} = v r_{f}$ for all $ f \in T$.

Now suppose that  $f \in T \cap J_{(u,v)}$, then $(u,v) \in f $ and $$(a_{f},a_{f}) = (u,v) r_{f} \, \in  f A_{f}$$  a contradiction to our assumption, hence $(2)\Rightarrow (3)$.

$(3) \Rightarrow (1) $ Given that $(3)$ holds, we give an explicit set of sentences that
axiomatises $\mathcal{EP}$.

For any element $\rho =(s,t)\in S \times S$ with $sa  \leq  ta$, for some $a \in A$ where $A \in \mathcal{EP}$, we choose and fix a finite set of
elements $\{(u_{\rho  1},v_{\rho 1}) \cdots (u_{\rho  n(\rho)},v_{\rho  n(\rho)}) \} $ of 
${\mathbf{R}^{\leq}(\rho)}$ as guaranteed by $(3)$. We define sentences $\phi_\rho$ of $L_S$ as follows:

If $sa \not \not\leq ta$ for all $ a \in A \in \mathcal{EP}$, let

\begin{equation*}
\phi_{\rho}=(\forall x)(sx \not  \leq  tx);
\end{equation*}

otherwise,   
\begin{equation*}
\phi_\rho=(\forall x)\big(sx  \leq tx \rightarrow (\exists z)(\bigvee
^{n(\rho)}_{i=1}(x=u_{\rho i}z\, = v_{\rho i}z ))\big) .
\end{equation*}

\vspace{2mm}

Let 
\begin{equation*}
\sum_{\mathcal{EP}} = \big\{\phi_{\rho}:\rho \in S \times S \big\}
\end{equation*}

We claim that $\sum_{\mathcal{EP}}$ axiomatises the class $\mathcal{EP}$.

Suppose that $A\in \mathcal{EP}$ and $\rho=(s,t) \in S
\times S$. If   $sb \not \leq  tb $, for all $b \in B \in \mathcal{EP}$, then certainly this is true for $A$, so that $A \models \phi_{\rho}.$

Suppose on the other hand that $sb \leq  tb $, for some $b \in B \in \mathcal{EP}$;
then 
$$\phi_{\rho}=(\forall x)\big(sx  \leq tx \rightarrow (\exists z)(\bigvee
^{n(\rho)}_{i=1}(x=u_{\rho i}z\, = v_{\rho i}z ))\big).$$

Suppose  $sa \leq  ta $ where $a
\in A $. As $A\in \mathcal{EP}$, $(3)$ tells us that  there is
an element $b\in A$ and $(u_{\rho i},v_{\rho i})$ for
some $i\in \{ 1,\hdots ,n(\rho)\}$ with 
$a=u_{\rho i}b=v_{\rho i}b$.  Hence $A \models \phi_{\rho}$.

Conversely suppose that $A$ is a model of $\sum_{\mathcal{EP}}$ and $s a  \leq  ta$ where $s,t \in S$ and $a \in
A$.  We cannot have that  $\phi_{\rho}$ is  $(\forall x) (sx \not  \leq tx)$. It follows that for some $ b \in B \in \mathcal{EP}$ we have $ sb  \leq  tb$, $ f =\{(u_{\rho\,1}, v_{\rho\,1}),\cdots, (u_{\rho\,n(\rho)},v_{\rho\,n(\rho)})\}$ exists as in $(3)$ and $\phi_\rho$ is 
\begin{equation*}
(\forall x) \big(sx \, \leq  \,tx \,\rightarrow \,(\exists \,z) (
\bigvee ^{n(\rho)}_{i=1}(\,x\,=\,u_{\rho i}z \,=  \,v_{\rho i}
\,z))\big).
\end{equation*}
Hence there exists an element $c \in A$ with $a=u_{\rho i}  c = v_{\rho
i}  c$ for some $i \in \{1,2, \hdots, n(\rho) \} $.  By definition of $
u_{\rho i}, v_{\rho i}$ we have $su_{\rho i}\,  \leq  t v_{\rho i } $. 
Thus $A$ satisfies Condition (EP) and so $\sum_{\mathcal{EP}}$ axiomatises $\mathcal{EP}$.
\end{proof}

\subsection{Axiomatisability of Condition (PWP)}\label{subsec:pwp}

We solve the axiomatisability problem for $\mathcal{PWP}$
 by  following similar lines to those  for $\mathcal{EP}$.

\begin{thm}\label{thm:(pwp)} 
The following conditions are equivalent for a  pomonoid $S$:

(1) the class $\mathcal{PWP}$ is axiomatisable;

(2) the class $\mathcal{PWP}$ is closed under ultraproducts;

(3)  for any $s \in S$ either $sa \not \leq sa'$ for all $ a,a' \in A \in \mathcal{PWP}$ or there exists a finite subset $f$ of $ \mathbf{R}^{\leq}(s,s)$, such that 
for any $a,a' \in A \in \mathcal{PWP}$
\[ sa  \leq  sa' \Rightarrow (a,a')=(u,v)b\mbox{ for some }(u,v) \in f,b \in A.\]

\end{thm}

\subsection{Axiomatisability of Condition (P$_w$) }\label{subsecP_W}

We recall that a left $S$-poset $A$ satisfies Condition (P$_w$) if for any $a,a'\in A$
and $s,s'\in S$, 
if $sa\leq s'a'$, then there exists $a''\in A$ and $u,u'\in S$ with
$(u,u')\in R^{\leq}(s,s')$, $a\leq ua''$ and $u'a''\leq a'$.

\begin{thm}\label{thm:pwconditionsposet} The
following conditions are equivalent for a pomonoid $S$:

(1) the class $\mathcal{P}_W$ is axiomatisable;

(2) the class $\mathcal{P}_W$ is closed under
 ultraproduct; 

(3) every ultrapower of $S$ satisfies Condition (P$_w$);

(4) for any $\rho=(s, t) \in S\times S$, either
$R^{\leq}(s,t)=\emptyset$ or there exists finitely many 
\[(u_{\rho 1},v_{\rho 1}),\hdots, (u_{\rho n(\rho)},v_{\rho n(\rho)})\in R^{\leq}(s,t)\]
such that for any $(x,y)\in R^{\leq}(s,t)$,
\[x\leq u_{\rho i}h\mbox{ and }v_{\rho i}h\leq y\]
for some $i\in \{ 1,\hdots n(\rho)\}$ and $h\in S$.

\end{thm}
\begin{proof}
$(3) \Rightarrow (4)$ Suppose that every ultrapower of $S$ has (P$_w$) but that
$(4)$ does not hold. Then there exists $\rho=(s,t)\in R^{\leq}(s,t)$ with
$R^{\leq}(s,t)\neq\emptyset$ but such that no finite subset of
$R^{\leq}(s,t)$ exists as in $(4)$.

Let 
$\{(u_\beta, v_\beta): \beta < \gamma \}$ be a
set of minimal (infinite) cardinality $\gamma$ contained
in $R^{\leq}(s,t)$ such that if $(x,y)\in R^{\leq}(s,t)$, then
\[x\leq u_{\beta}h\mbox{ and }v_{\beta}h\leq y\]
for some $\beta<\gamma$ and $h\in S$. From the minimality of
$\gamma$ we may assume that for any $\alpha<\beta<\gamma$, it is not true that both
\[u_{\beta}\leq u_{\alpha}h\mbox{ and }v_{\alpha}h\leq v_{\beta}\]
for any $h\in S$.

Let $\Phi$ be a uniform ultrafilter on $\gamma$, that is $\Phi$ is an ultrafilter on $\gamma$ such that all sets in $\Phi$ have cardinality $\gamma$.
Let $\mathcal{U}= S ^{\gamma}/ \Phi$, by assumption $\mathcal{U}$ satisfies Condition (P$_w$).

Since $su_{\beta}\leq tv_{\beta}$ for all $\beta<\gamma$, $s(u_{\beta})_{\Phi}
\leq t(v_{\beta})_{\Phi}$. As $\mathcal{U}$ satisfies condition (P$_w$), 
there exists $(u,v)\in R^{\leq}(s,t)$ and
$(w_{\beta})_{\Phi}\in \mathcal{U}$ such that
\[(u_{\beta})_{\Phi}\leq u(w_{\beta})_{\Phi}
\mbox{ and }v(w_{\beta})_{\Phi}\leq (v_{\beta})_{\Phi}.\]
Let $D\in\Phi$ be such that 
\[u_{\beta}\leq uw_{\beta}\mbox{ and }vw_{\beta}\leq v_{\beta}\]
for all $\beta\in D$. Now
$(u,v)\in R^{\leq}(s,t)$ so that
\[u\leq u_{\sigma}h\mbox{ and }v_{\sigma}h\leq v\]
for some $\sigma<\gamma$.  Choose $\beta\in D$ with $\beta>\sigma$. Then
\[u_{\beta}\leq uw_{\beta}\leq u_{\sigma}hw_{\beta}\mbox{ and }
v_{\sigma}hw_{\beta}\leq vw_{\beta}\leq v_{\beta},\]
a contradiction. Thus $(4)$ holds.

$(4) \Rightarrow (1)$ Suppose that $(4)$ holds.

Let $\rho =(s,t) \in S \times S$. If
$R^{\leq}(s,t)=\emptyset$ we put 
\[\Omega_{\rho}:=(\forall x)(\forall y)(sx \not \leq ty).\]
If $R^{\leq}(s,t)\neq\emptyset$, let
\[(u_{\rho 1},v_{\rho 1}),\hdots, (u_{\rho n(\rho)},v_{\rho n(\rho)})\in R^{\leq}(s,t)\]
be the finite set given by our hypothesis, and put 
$$\Omega_{\rho}:=( \forall x)(\forall y)\big( sx \leq ty \rightarrow (\exists z)( \bigvee ^{n(\rho)}_{i=1} ( x \leq  u_{\rho,i} \,z \wedge   v_{\rho,i}\, z \leq   y ))\big.$$

Let \[\sum_{\mathcal{P}_W}=\{ \Omega_{\rho}: \rho\in S\times S\}.\]

We claim that $ \sum _{\mathcal{P}_W}$ axiomatises 
$\mathcal{P}_W$.

Let $A\in\mathcal{P}_w$ and let
$\rho=(s,t)\in S\times S$. Suppose first that 
$R^{\leq}(s,t)=\emptyset$. If $sa\leq tb$ for some $a,b\in S$, then as
$A$ satisfies (P$_w$) we have, in particular, that
$R^{\leq}(s,t)\neq\emptyset$, a contradiction. Hence  
$A\models\Omega_{\rho}$.

On the other hand, if $R^{\leq}(s,t)\neq\emptyset$, then
\[\Omega_{\rho}=  (\forall x)(\forall y)\big( sx \leq ty \rightarrow (\exists z)( \bigvee ^{n(\rho)}_{i=1} ( x \leq  u_{\rho,i} \,z \wedge   v_{\rho,i}\, z \leq   y ))\big.\]
If $sa\leq tb$ where $ a,b \in A$, then there exists $(u,v)\in R^{\leq}(s,t)$ and
$c\in A$ with
\[a\leq uc\mbox{ and }vc\leq b.\]
By hypothesis we have that
\[u\leq u_{\rho i}h\mbox{ and }v_{\rho i}h\leq v\]
for some $h\in S$ and $i\in \{ 1,\hdots ,n(\rho)\}$. Now
\[a\leq uc\leq u_{\rho i}hc\mbox{ and } v_{\rho i}hc\leq vc\le b\]
so that (with $z=hc$), $A\models\Omega_{\rho}$. Hence $A \models 
\sum_{\mathcal{P}_W}$.

Conversely, suppose that  $A\models \sum_{\mathcal{P}_W}$ and $sa\leq tb$ for some
$\rho=(s,t)\in S\times S$ and $a,b\in S$. We must therefore have that
$R^{\leq}(s,t)\neq \emptyset$ and consequently, $\Omega_{\rho}$ is 
\[(\forall x)(\forall y)\big( sx \leq ty \rightarrow (\exists z)( \bigvee ^{n(\rho)}_{i=1} ( x \leq  u_{\rho,i} \,z \wedge   v_{\rho,i}\, z \leq   y ))\big.\]
Hence $a\leq u_{\rho i}c$ and $v_{\rho i}c\leq b$ for some $c\in A$. By definition,
$(u_{\rho i},v_{\rho i})\in R^{\leq}(s,t)$, so that $A$ lies in 
$\mathcal{P}_W$. 
\end{proof}

\subsection{Axiomatisability of Condition (PWP$_{w}$)}\label{subsec:pwpw}

We solve the axiomatisability problem for Condition (PWP$_{w}$) by following  similar lines 
to those  for Condition (P$_w$).   Of course in this case $R^{\leq}(s,s)\neq \emptyset$ for any $s\in S$ and so our result is as follows.

\begin{thm}\label{thm:pwpw} The
following conditions are equivalent for a pomonoid $S$:

(1) the class $\mathcal{PWP}_w$ is axiomatisable;

(2) the class $\mathcal{PWP}_w$ is closed under
 ultraproduct; 

(3) every ultrapower of $S$ satisfies Condition (PWP$_w$);

(4) for any $s\in S$ there exists finitely many 
\[(u_{\rho 1},v_{\rho 1}),\hdots, (u_{\rho n(\rho)},v_{\rho n(\rho)})\in R^{\leq}(s,s)\]
such that for any $(x,y)\in R^{\leq}(s,s)$,
\[x\leq u_{\rho i}h\mbox{ and }v_{\rho i}h\leq y\]
for some $i\in \{ 1,\hdots n(\rho)\}$ and $h\in S$.

\end{thm}

\subsection{Axiomatisability of Condition $(W)$}\label{sec:axiocondwsposet}
 
 For our final class defined by an interpolation condition, we consider
 $\mathcal{W}$.

\begin{thm}

The following conditions are equivalent for an pomonoid $S$:

(1) the class $\mathcal{W}$ is axiomatisable;

(2) the class $\mathcal{W}$ is closed under ultraproducts; 

(3) every ultrapower of $S$ lies in $\mathcal{W}$;

(4) for any $s, t \in S$ there exists   an integer $n\geq 0$, $$p_1,\cdots,p_n\in sS\mbox{ and }q_1,\cdots,q_n\in tS$$ such that for
all $i\in\{ 1,\hdots ,n\}$ we have $ p_i \leq q_i$, and if  $ s u \leq tv $ then there
exists $i \in \{ 1, \cdots, n \} $ and $ z \in S$ with  
\[su \leq  p_i z\mbox{ and }
 q_i z \leq tv .\]

\end{thm}
\begin{proof}
$(3) \Rightarrow (4)$ Suppose that every ultrapower of $S$ has
(W) but that $(4)$ fails. Then there exists $s,t \in S$ such that there does not exist
any finite list $p_1,\hdots, p_n,q_1,\hdots,q_n$ satisfying the conditions
of $(4)$.

Let $ \gamma$ be a cardinal minimal with respect to the existence of a set $\{(u_{\beta},v_{\beta}): \beta < \gamma \}$ such that $ u_{\beta} \in s S, \,  v_{\beta} \in t S \, , u_{\beta} \leq v_{\beta} $ and if $  su \leq tv $ then 
there exists $ \beta<\gamma$ and $ z \in S$ with $ su \leq u_{\beta} z, \, v_{\beta} z  \leq t v $. 

Certainly $\gamma $ exists since we could consider $  \{ ( sx , ty ): x, \, y \in S, sx \leq ty \}$. We are assuming that $ \gamma $ is infinite. By the minimality of $\gamma$ we
can assume that it is not true that for any $\gamma >\beta>\sigma$, we have both
$ u_{\beta} \leq u_{\sigma} \, k $ and $ v_{\sigma} \, k \leq v_{\beta} $.

Let $\Phi$ be a uniform ultrafilter on $\gamma$ and let $\mathcal{U}= S ^{\gamma}/ \Phi$; by assumption $\mathcal{U}$ satisfies Condition (W).

Since each $u_{\beta}\in sS, u_{\beta}=sx_{\beta}$ for some $x_{\beta}\in S$; similarly,
$v_{\beta}=ty_{\beta}$ for some $y_{\beta}\in S$. Now
$u_{\beta}\leq v_{\beta}$ for all $\beta<\gamma$, so that
$s(x_{\beta})_{\Phi}\leq t(y_{\beta})_{\Phi}$ and as
$\mathcal{U}$ has (W), there exists $(w_{\beta})_{\Phi}\in\mathcal{U}$,
$p\in sS$ and $q\in tS$ with
\[p\leq q, s(x_{\beta})_{\Phi}\leq p(w_{\beta})_{\Phi}\mbox{ and }
q(w_{\beta})_{\Phi}\leq t(y_{\beta})_{\Phi}.\]
Let $D\in \Phi$ be such that
\[sx_{\beta}\leq pw_{\beta}\mbox{ and }qw_{\beta}\leq ty_{\beta}\]
for all $\beta\in D$. As $p\leq q$ there exists $\sigma<\gamma$ and $z\in S$ with
\[p\leq u_{\sigma}z\mbox{ and }v_{\sigma}z\leq q.\]
Hence, choosing $\beta\in D$ with $\beta>\sigma$,
\[u_{\beta}=sx_{\beta}\leq pw_{\beta}\leq u_{\sigma}zw_{\beta}
\mbox{ and }v_{\sigma}zw_{\beta}\leq qw_{\beta}\leq ty_{\beta}=v_{\beta},\]
a contradiction. Hence $(4)$ holds.

$(4) \Rightarrow (3)$ Suppose now that $(4)$ holds. For each
$\rho=(s,t)\in S\times S$ let
\[p_{\rho 1},\hdots ,p_{\rho n(\rho)}, q_{\rho 1},\hdots ,q_{\rho n(\rho)}\]
be the list of elements of $S$ guaranteed by $(4)$. If
$n(\rho)=0$, let
\[\Omega_{\rho}:(\forall x)(\forall y)(sx\not\leq ty).\]
If
$n(\rho)\geq 1$, let
$$\Omega_{\rho}:=( \forall x)(\forall y)\big( sx \leq ty \rightarrow (\exists z)( \bigvee ^{\rho(n)}_{i=1} ( s x \leq p_{\rho i} z \wedge  q_{\rho i}z \leq  t y ))\big)$$
and let
\[\sum_{\mathcal{W}}= \{ \Omega_{\rho}:\rho\in S\times S\}.\]

We claim that $ \sum_{\mathcal{W}} $ axiomatises $\mathcal{W}$.

Let $A\in\mathcal{W}$ and $\rho=(s,t)\in S\times S$. If $n(\rho)=0$ and
$sa\leq tb$, for some $a,b\in A$, then, in particular, $su\leq tv$ for some
$u,v \in S$. By as $(4)$ holds this gives that $n(\rho)\geq 1$, a contradiction.
Hence $A\models \Omega_{\rho}$.

Suppose now that $n(\rho)\geq 1$, so that
$$\Omega_{\rho}=( \forall x)(\forall y)\big( sx \leq ty \rightarrow (\exists z)( \bigvee ^{\rho(n)}_{i=1} ( s x \leq p_{\rho i} z \wedge  q_{\rho i}z \leq  t y ))\big).$$
If $sa\leq tb$ for some $a,b\in A$, then there exists $p\in sS,q\in tS$ 
and $c\in A$ such that
\[p\leq q, sa\leq pc\mbox{ and }qc\leq tb.\]
By $(4)$, \[p\leq p_{\rho i}z\mbox{ and }q_{\rho i}z\leq q\]
for some $i\in\{ 1,\hdots,n(\rho)\}$ and $z\in S$. Hence
\[sa\leq p_{\rho i}zc\mbox{ and }q_{\rho i}zc\leq tb\]
so that $A\models \Omega_{\rho}$. Hence $A\models \sum_{\mathcal{W}}$.

Conversely, if $A\models \sum_{\mathcal{W}}$ and $sa\leq tb$ for some
$\rho=(s,t)\in S\times S$ and $a,b\in A$, then we must have $n(\rho)\geq 1$ and
$$\Omega_{\rho}=( \forall x)(\forall y)\big( sx \leq ty \rightarrow (\exists z)( \bigvee ^{\rho(n)}_{i=1} ( s x \leq p_{\rho i} z \wedge  q_{\rho i}z \leq  t y ))\big).$$
Then 
\[sa\leq p_{\rho i}c\mbox{ and }q_{\rho i}c\leq tb\]
for some $i\in \{ 1,\hdots ,n(\rho)\}$ and $c\in A$. By choice of $p_{\rho i},q_{\rho i}$, we see
that $A\in\mathcal{W}$. Hence $\sum_{\mathcal{W}}$ axiomatises $\mathcal{W}$ as required.

\end{proof}

\section{Axiomatisability of Projective and Free $S$-posets}\label{sec:projfree}

The axiomatisability problems for $\mathcal{P}r$ and $\mathcal{F}r$ are easily solved
from the results of the previous section and the answers to the corresponding questions in
the $S$-act case. 

\subsection{Axiomatisability of $\mathcal{P}r$}\label{subsec:p}

The question of the axiomatisability of $\mathcal{P}r$ was addressed
in \cite{step:2009}. Without giving much detail, Pervukhin and Stepanova
indicate that if every ultrapower of a pomonoid $S$ is projective as 
a left $S$-poset, then it can be argued, following the corresponding proofs
for $S$-acts, that $S$ is poperfect, which here can be taken to mean $\mathcal{SF}=\mathcal{P}r$ in the class of left $S$-posets. In \cite{step:2009}
this is then utilised to show that $\mathcal{P}r$ is axiomatisable if and only if
$\mathcal{SF}$ is axiomatisable and $\mathcal{SF}=\mathcal{P}r$. 
 Notice that
in \cite{step:2009}, the classes $\mathcal{SF}$ and $\mathcal{P}r$ are
denoted by $\mathcal{SF}^{<}$ and $\mathcal{P}^{<}$, to distinguish them from
the classes of strongly flat and projective left $S$-{\em acts}, a convention we have not followed here.  

The current authors have
shown that a pomonoid $S$ is left perfect as a {\em monoid} if and only if it is left
perfect as a {\em pomonoid} \cite{gouldshaheen}. With this in mind we can give a short and direct proof of the following.

\begin{thm}\label{thm:proj} The following are equivalent for a pomonoid $S$:

(1) the class $\mathcal{P}r$ is axiomatisable;

(2) every ultrapower of $S$ is projective as a left $S$-poset;

(3) the class $\mathcal{SF}$ is axiomatisable and $\mathcal{SF}=\mathcal{P}r$.

\end{thm}
\begin{proof} Clearly we need only prove that $(2)$ implies $(3)$; suppose that $(2)$ holds. 

Let $\mathcal{U}= S^{\gamma}/\Phi$ be an ultrapower of $S$ {\em as a left $S$-act}, then $$\mathcal{U}= \prod_{\gamma \in \wedge} S^{\gamma}/ \equiv$$ where $$(a_i) \equiv (b_i) \Leftrightarrow \{ i: a_i =b_i \} \in \Phi$$
and $$s(a_i)_{\Phi} = (sa_i)_{\Phi} \,\,\mbox {is  a well-defined $S$-action}.  $$

Consider the corresponding ultrapower of $S$ as a left $S$-poset, that is, $\mathcal{U}^{'}= S^{\gamma}/\Phi$. Here $\equiv$ and the $S$-action are defined as before and $$(a_i)_{\Phi} \leq (b_i)_{\Phi} \Leftrightarrow \{ i : a_i \leq b_i \} \in \Phi\hfill{\phantom{mm}}(*).$$

In other words $\mathcal{U}$ is $\mathcal{U}^{'}$ equipped with the partial order defined as in $(*)$.

We are supposing $\mathcal{U}^{'}$ is projective as a left $S$-poset, that is, there exists a disjoint union ${\bigcup_{i \in I}} S e_{i}$ where $e_i$s are idempotents, and an $S$-po-isomorphism $\theta:\mathcal{U}^{'} \to \bigcup_{i \in I} S e_i$. Regarding $\bigcup_{i \in I} S e_i$ as an $S$-act, $\theta:\mathcal{U}\rightarrow
\bigcup_{i\in I}Se_i$  is certainly an
$S$-act isomorphism.   We can conclude that every ultrapower of $S$ as a left $S$-act is projective. From \cite[Theorem 8.6]{gouldpalyutin},
$S$ is left perfect, so from \cite[Theorem 6.3]{gouldshaheen},
$S$ is left poperfect. Hence $\mathcal{SF}=\mathcal{P}r$. From \cite[Theorem 4.8]{step:2009} and \cite{shaheen:2010}, we also have that $\mathcal{SF}$ is axiomatisable.
Hence $\mathcal{P}r$ is axiomatisable.
\end{proof}

\subsection{Axiomatisability of $\mathcal{F}r$.}\label{subsec:free}

To explain our result we need to recall the following definition from \cite{gould:tartu}.
Let $e  \in E(S)$, where $E(S)$ is the set of idempotents of a monoid $S$, and let $a \in S$. We say that $a  = xy$  is an {\em $e$-good factorisation  of $a$ through $x$} if $ y \not = wz$ for any $w,z$ with $e= xw$ and $e\, \mathcal{L}\, w$
(see \cite{gould:tartu}).

\begin{thm}\label{thm:free} The following conditions are equivalent for a pomonoid $S$:

(1) every ultrapower of left $S$-poset $S$ is free;

(2) $\mathcal{P}r$ is axiomatisable and $S$ satisfies $(*)$: for all $e \in E(S) \setminus {1}$, there exists a finite set $f \in S$ such that any $ a \in S$ has an $e$-good factorization through $x$, for some $x \in f$;

(3) the class $\mathcal{F}r$ is axiomatisable. 
\end{thm}
\begin{proof} $(1)\Rightarrow (2)$ Since every ultrapower of $S$ is free as a left $S$-poset, it is free as a left $S$-act with the same argument as in Theorem~\ref{thm:proj}. By \cite[Theorem 5.3]{gould:tartu}, $S$ satisfies $(*)$. Also by 
Theorem~\ref{thm:proj}, $\mathcal{P}r$ is axiomatisable.

$(2)\Rightarrow (3)$ If $\mathcal{P}r$ is axiomatisable, then every ultrapower of copies of 
$S$ is projective as a left $S$-poset, and hence as a left $S$-act. From
\cite[Lemma 8.4]{gouldpalyutin}, it follows that for any $e\in E(S)$ and $u\in S$, there are
only finitely many $x\in S$ such that $e=ux$. This permits us to define the sentences
$\varphi_e$ as in \cite{gouldpalyutin}. Let
$\sum_{\mathcal{P}r}$ be the set of sentences axiomatising the projective left
$S$-posets. Then, as in \cite[Theorem 9.1]{gouldpalyutin}, 
\[\sum_{\mathcal{P}r}\cup\big\{ \varphi_e:e\in E(S)\setminus\{ 1\}\big\}.\]
axiomatises $\mathcal{F}r$.

\end{proof}

\section{Some Open Problems}\label{sec:open}

We aim to axiomatise the class of left $S$-posets satisfying 
Condition (WP), and (WP)$_{w}$. The finitary conditions that arise in
axiomatising classes of $S$-posets, as in the case for $M$-acts, are
related
to more standard finitary conditions such as chain conditions. We aim to
investigate these connections, particularly in the context of invese
monoids equipped with the natural partial order. For examples in the
unordered case, we refer the reader to \cite{gould:tartu}.

\end{document}